\documentclass[10pt]{amsart}
\usepackage{geometry}                
\usepackage{graphicx}
\usepackage{amssymb}
\usepackage{epstopdf}
\usepackage{multicol}
\newtheorem{theorem}{Theorem}
\newtheorem{prop}[theorem]{Proposition}

\newtheorem{lemma}[theorem]{Lemma}

\newtheorem{remark}[theorem]{Remark}

\def \bb {\beta}
\def \gg {\gamma}
\def \dd {\delta}
\def \TT {\tau}
\def \bu {\bullet}
\def \SRA {\, \longrightarrow }
\def\multi#1{\vbox{\baselineskip=0pt\halign{\hfil$\scriptstyle\vphantom{(_)}##$\hfil\cr#1\crcr}}}

\def \BB {{\bf B}}

\def\tttt #1{{\textstyle{#1} }}

\def \DA {\downarrow\hskip -.04in}

\def \magstep#1 {\ifcase#1 1000\or 1200\or 1440\or 1728\or 2074\or 2488\fi\relax}

\def\la{{\lambda}}

\def \con {\subseteq}

\def\sig{\sigma}

\def \ggg {\gamma}
\def \GG {\Gamma}

\def \-> {\rightarrow}
\def\LL{\big\langle}

\def\RR {\big\rangle}

\def\OM {\Omega}

\def\om {\omega}
\def\la {\lambda}
\def\La {\Lambda}
\def \RA {\rightarrow}
\def \LA {\lefttarrow}

\def \sas {\vskip .06truein}
\def\sa{{\vskip .125truein}}

\def \eee {\epsilon}

\def\aaa {\alpha}
\def\bbb {\beta}

\def\ggg {\gamma}
\def\aa {\alpha}
\def\bb {\beta}
\def\dd{\delta}
\def\gg {\gamma}
\def\con {\subseteq}

\def \ses {\enskip = \enskip}
\def \sps {\, + \,}

\def \sms {\, - \,}

\def \shhh {{\scriptstyle \cup\hskip -.18em\cup}}

\def \scs {\, , \,}
\def \ess {\enskip}

\def \ssp {\hskip .25em}
\def \bigsp {\hskip .5truein}
\def \part {\vdash}

\def \OM {{\Omega}}

\def \RA {{ \rightarrow }}
\def \LA {{ \leftarrow }}

\def \om {\omega}

\def \TH {{\tilde H}}

\def \om {\omega}

\def \TH {{\tilde H}}

\def \scs {\ssp , \ssp}
\def \ess {\enskip}
\def \ssp {\hskip .25em}
\def \bigsp {\hskip .5truein}
\def \part {\vdash}

\def \BB {{\bf B}}
\def \BC {{\bf C}}
\def \BZ {{\bf Z}}

\def \CPF {{\mathcal PF}_n}

  \def \BB {{\bf B}}
 \def \TBH {  {\bf C}}

\def\middle{{\it middle}}
\def\abig{{\it big}}
\def\asmall{{\it small}}

\numberwithin{equation}{section}
\numberwithin{theorem}{section}

\title[Three Shuffle Parking Functions]{ A three shuffle case of the compositional parking function conjecture}
\author{Adriano M. Garsia, Guoce Xin and Mike Zabrocki}

\begin{document}

\begin{abstract}
We prove here that the polynomial  $\LL\nabla  \BC_p\, 1\scs  e_ah_b h_c\RR$
$q,t$-enumerates, by the statistics {\it dinv} and {\it area}, 
the parking functions whose supporting Dyck path touches the main diagonal 
according to the composition $p\models a+b+c$ and have a reading word which 
is a shuffle of one decreasing word and two increasing words of respective sizes $a,b,c$.
Here $\BC_p\, 1$ is a rescaled Hall-Littlewood polynomial
and $\nabla$ is the Macdonald eigen-operator introduced by Bergeron and Garsia \cite{BG:1999}. 
This is our latest 
progress in  a continued effort to settle the decade old {\it shuffle conjecture}
of Haglund et. al. \cite{HHLRU:2005}.
This result includes as special cases all previous results connected with the shuffle conjecture
such as the $q,t$-Catalan \cite{GH:2002} and the Schr\"oder and $h,h$ results of Haglund in \cite{Haglund:2004} 
as well as their compositional refinements recently obtained by the authors in \cite{GXZ:2011} and \cite{GXZ}. 
It also confirms the possibility that the approach adopted in \cite{GXZ:2011} and \cite{GXZ} has the potential 
to yield a resolution of the shuffle parking function conjecture as well as its 
compositional refinement more recently proposed by Haglund, Morse and Zabrocki in \cite{HMZ:2012}.
\end{abstract}
\maketitle

\begin{section}{Introduction}

A parking function may be visualized as a Dyck path in an 
$n \times n$ lattice square with the cells adjacent to the vertical edges of the path 
labelled with a permutation of the integers $\{1,2, \ldots, n\}$ in a column increasing 
way (see attached figure). We will borrow from parking function language by
calling these labels {\it cars}.
The corresponding preference function is simply obtained by 
specifying that {\it car } $i$ prefers to park at the bottom of its column. 
This visual representation, which  has its origins in  \cite{GH:1996}, uses the Dyck path  
to assure that the resulting preference function  parks the cars. 

\noindent
\begin{minipage}[l]{0.70 \textwidth}

\hskip .2in The sequence of cells 
that joins the SW corner of the lattice square to the NE corner will be called  
the {\it main diagonal} or the  $0$-diagonal of the parking function. The successive diagonals 
above the main diagonal will be referred to as diagonals $1,2,\ldots, n-1$ respectively. 
On the left of  the adjacent display we have listed the diagonal numbers of the 
corresponding cars. It is also convenient to represent a parking function as  a two line array
\begin{equation}
PF = 
\begin{bmatrix}
v_1 & v_2 & \cdots  &  v_n\\
u_1 & u_2 & \cdots  &  u_n
\end{bmatrix}
\label{eq:II.1}
\end{equation}
\end{minipage}
\hfill$
\vcenter{\hbox{\includegraphics[width=1.7in]{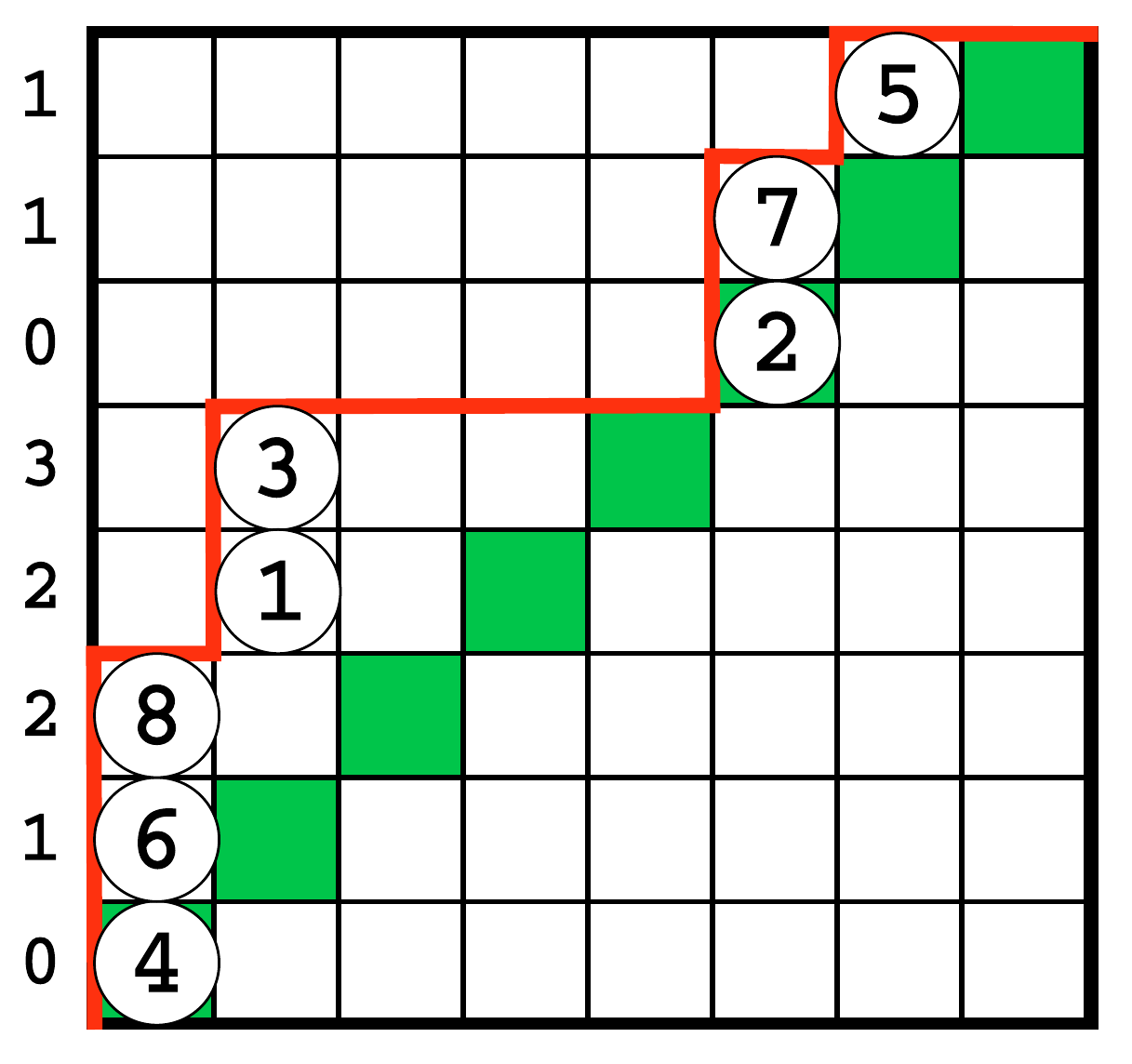}}}$ 

\noindent
where $v_1,v_2,\ldots ,v_n$ are the cars as we read them by  rows from bottom to top
and $u_1,u_2,\ldots ,u_n$ are their corresponding diagonal numbers. From the geometry of the above display we can immediately  see that for a two line array to represent a parking function it is necessary and sufficient that we have
\begin{equation}
u_1=0\ess\ess\ess \hbox{and }\ess\ess\ess\ess 0\le u_i\le u_{i-1}+1
\label{eq:II.2}
\end{equation}
with $V=(v_1 , v_2 , \ldots  ,  v_n)$ a permutation satisfying the condition
\begin{equation}
u_i=u_{i-1}+1\ess\Longrightarrow\ess v_i>v_{i-1}~.
\label{eq:II.3}
\end{equation}
For instance the parking function in the above display corresponds to the following two line array
$$
PF=\begin{bmatrix}
4 &  6 &  8 &  1 &  3  & 2 &  7 &  5 \\
0 &  1 &  2 &  2 &  3  & 0 &  1 &  1 
\end{bmatrix}
$$ 

Parking functions will be enumerated here by means of a weight that is easily defined in terms of their two line arrays. To this end, let us denote by $\sig(PF)$ the permutation 
obtained by successive right to left readings of the components 
of  $V=(v_1 , v_2 , \ldots  ,  v_n)$ according to decreasing values of
$u_1 , u_2 , \ldots  ,  u_n$. We will call $\sig(PF)$ the
{\it diagonal word} of $PF$. We will also let $ides(PF)$ 
denote the descent set of the inverse of $\sig(PF)$. 

\hskip .2in This given, each parking function is assigned the weight  
\begin{equation}
w(PF)\ses t^{area(PF)} q^{dinv(PF)} Q_{ides(PF)}[X]
\label{eq:II.4}
\end{equation}

where 
\begin{equation}
area(PF)\ses \sum_{i=1}^n u_i
\label{eq:II.5}
\end{equation}

\noindent 

\begin{equation}
dinv(PF)= \hskip -.07in
\sum_{1\le i<j\le n } \chi( u_i=u_j\, \&\, v_i<v_j)  +\hskip -.07in
\sum_{1\le i<j\le n }\hskip -.07in \chi( u_i=u_j+1\, \&\, v_i>v_j)
\label{eq:II.6}
\end{equation}
and for a subset $S\con\{1,2, \ldots, n-1\}$, $Q_S[X]$ denotes  Gessel's \cite{Gessel:1984}
fundamental quasi-symmetric function.

These statistics also have  a  geometrical meaning and can  be directly obtained from the visual representation. In fact we can easily see that 
 the diagonal numbers give  the number of lattice cells 
between the Dyck path and the main diagonal in the row of each corresponding car.
Thus  the sum in \eqref{eq:II.5} gives the total number of cells between the supporting Dyck path and 
the main diagonal.

It is also easily seen that two cars in the same diagonal with the car on the left smaller 
than the car on the right will contribute a unit to $dinv(PF)$ called {\it a primary diagonal inversion}.
Likewise,  
a car on the left that 
is bigger than a car on the right with the latter in the adjacent lower diagonal

\noindent
$\vcenter{\hbox{
\includegraphics[width=1.6in]{park2.pdf}}}$\hskip .15in
\begin{minipage}{.70 \textwidth}
contributes a unit to $dinv(PF)$ called {\it a secondary diagonal inversion}.
Thus the sum in \eqref{eq:II.6} gives the total number of {\it diagonal inversions} of the parking function.

Note that reading the cars  by diagonals from right to left
starting with the highest diagonal we see that car $3$ is in the third diagonal,
$1$ and $8$ are in the second 
diagonal, $5,7$ and $6$ are in the first diagonal and $2$ and $4$ are in 
the main diagonal. This gives

\begin{equation}
\sig(PF)\ses 3\,1\,8\,5 \,7\,6 \,2\,4
\ess\ess\ess\hbox{and}\ess\ess\ess 
 ides(PF)=\{2,4,6,7\}
\label{eq:II.7}
\end{equation}
\end{minipage}

\noindent
Thus  for the parking function given  above we have 
$$
 area(PF)=10,\ess dinv(PF)=4,
$$
which   together with \eqref{eq:II.7} gives
$$
w(PF)\ses t^{10}q^4 Q_{\{2,4,6,7\}}[X].
$$
In \cite{HMZ:2012}, Haglund, Morse and Zabrocki introduce an additional  statistic, the {\it diagonal composition} of a parking function,
which we denote by $p(PF)$. This is  the composition whose parts
determine the position of the  zeros in the vector  $U=(u_1,u_2,\ldots ,u_n)$,
or equivalently give the lengths of the segments between successive diagonal  touches of its Dyck path. Thus $p(PF)=(5,3)$ for the above example.

Denoting by $\CPF$  the collection of parking functions
 in the $n\times n$ lattice square one of the conjectures in   \cite{HMZ:2012} states that
 for any $p=(p_1,p_2,\ldots ,p_k)\models n$ we have
\begin{equation}
 \nabla \BC_{p_1}\BC_{p_2}\cdots \BC_{p_k}1\ses\hskip -.2in
\sum_{\multi {PF\in \CPF \cr p(PF)=(p_1,p_2,\ldots ,p_k)}}\hskip -.2in
t^{area(PF)} q^{dinv(PF)}Q_{ides(PF)}[X] 
\label{eq:II.8}
\end{equation} 
where $\nabla$ is the Bergeron-Garsia operator introduced in \cite{BG:1999} and, for each integer $a$,
$\BC_a$ is the  operator plethystically defined by setting for any symmetric function $P[X]$
\begin{equation}
 \BC_aP[X]\ses  \left(\frac{-1}{q}\right)^{a-1} \sum_{k\ge 0} P\left[X-\frac{1-1/q}{z}\right]
 z^k h_k[X]\Big|_{z^a}~.
\label{eq:II.9}
\end{equation} 
 It follows from a theorem of Gessel \cite{Gessel:1984} that the identity in \eqref{eq:II.9} is equivalent to the statement that  for any 
composition $(p_1,p_2,\ldots ,p_k)\models n$ and any partition 
$\mu=(\mu_1,\mu_2,\ldots ,\mu_\ell)\part n$
we have  
\begin{equation}
\LL\nabla \BC_{p_1}\BC_{p_2}\cdots \BC_{p_k}\, 1 \scs h_{\mu_1}h_{\mu_2}\cdots h_{\mu_\ell}\RR =\hskip -.2in
\sum_{\multi {PF\in \CPF \cr p(PF)=(p_1,p_2,\ldots ,p_k)}}\hskip -.2in
t^{area(PF) }q^{dinv(PF)}\chi(\sig(PF)\in 
E_1\shhh E_2\shhh \cdots \shhh E_\ell)
\label{eq:II.10}
\end{equation} 
where  
$E_1,E_2,\ldots ,E_\ell$ are successive segments 
of the word $1234\cdots n$ of respective lengths
$\mu_1,\mu_2,\ldots ,\mu_\ell$ and the symbol 
\hbox{$\chi(\sig(PF)\in 
E_1\shhh E_2\shhh \cdots \shhh E_\ell)$} is to indicate that the sum is to be carried out over parking functions in $\CPF$ whose diagonal word
is a shuffle of the words $E_1,E_2,\ldots,E_\ell$. In this paper we show that the symmetric function methods developed in 
\cite{BGHT:1999} and \cite{GHT:1999} can also be used to obtain the following identity
\sa

\begin{theorem} \label{thm:II.1}
For any triplet of integers $a,b,c\ge 0$ and compositions $p=(p_1,p_2,\ldots ,p_k)\models a+b+c$ we have  

\begin{equation}
\LL\nabla \BC_{p_1}\BC_{p_2}\cdots \BC_{p_k}\, 1 \scs  e_ah_{b} h_c\RR =
\sum_{\multi {PF\in \CPF \cr p(PF)=(p_1,p_2,\ldots ,p_k)}}
t^{area(PF) }q^{dinv(PF)}\chi\big(\sig(PF)\in\,\, 
 \DA E_1\, \shhh \,E_2\, \shhh E_3\,\big)
\label{eq:II.11}
\end{equation} 
where $E_1,E_2,E_3$ are successive segments 
of the word $1234\cdots n$ of respective lengths
$a,b,c$ and $\DA E_1$ denotes the reverse of the word $E_1$.
\end{theorem}

Since in \cite{HMZ:2012} it is shown that  
\begin{equation}
\sum_{p\models n}\BC_{p_1}\BC_{p_2}\cdots \BC_{p_k}1\ses e_n,
\label{eq:II.12}
\end{equation} 
summing \eqref{eq:II.11} over all compositions of $n$ we obtain that
\begin{equation}
\LL \nabla e_n\scs e_ah_{b}h_c\RR\ses
\sum_  {PF\in \CPF }
 t^{area(PF) }q^{dinv(PF)}\chi\big(\sig(PF)\in \,\,
 \DA  E_1\, \shhh \,E_2\, \shhh E_3\,\big).
\label{eq:II.13}
\end{equation} 

Setting $b=c=0$ in \eqref{eq:II.13} gives 
\begin{equation}
\LL \nabla e_n\scs e_n\RR\ses
\sum_  {PF\in \CPF }
 t^{area(PF) }q^{dinv(PF)}\chi\big(\sig(PF)=n\cdots 321
\,\big)~,
\label{eq:II.14}
\end{equation} 
which is the $q,t$-Catalan result of \cite{GH:2002}.
Setting $c=0$ or $a=0$ gives the following two identities
proved by Haglund in \cite{Haglund:2004}, Namely the Schr\"oder result
\begin{equation}
\LL \nabla e_n\scs e_k h_{n-k}\RR\ses
\sum_  {PF\in \CPF }
 t^{area(PF) }q^{dinv(PF)}\chi\big(\sig(PF)\in \,\,  \DA  A\, \shhh B\,
\,\big)
\label{eq:II.15}
\end{equation}  
and the shuffle of two segments result \cite{Haglund:2004} (see also
\cite{Haglund:2008} and \cite{GHS:2011})
\begin{equation}
\LL \nabla e_n\scs h_k h_{n-k}\RR\ses
\sum_  {PF\in \CPF }
 t^{area(PF) }q^{dinv(PF)}\chi\big(\sig(PF)\in   A\, \shhh  B\,
\,\big)
\label{eq:II.16}
\end{equation} 
with $A$ and $B$ successive segments  of respective lengths
$k$ and $n-k$. We should also note that setting $c=0$ and $a=0$ in \eqref{eq:II.11} gives the  two   
identities proved in \cite{GXZ:2011} and \cite{GXZ}.


It will be good at this point to exhibit  at least an instance of the identity in \eqref{eq:II.11}.
Below we have  the six parking functions with diagonal composition $(3,2)$
whose diagonal word is in the shuffle $1\shhh 23\shhh 45$.

Reading the display by rows, starting from the top row, we see that the first $PF$ has area $4$ and the remaining ones have area $3$. Their dinvs are respectively created by the pairs
of cars  listed below each parking function
$$
\vcenter{\hbox{\includegraphics[width=4in]{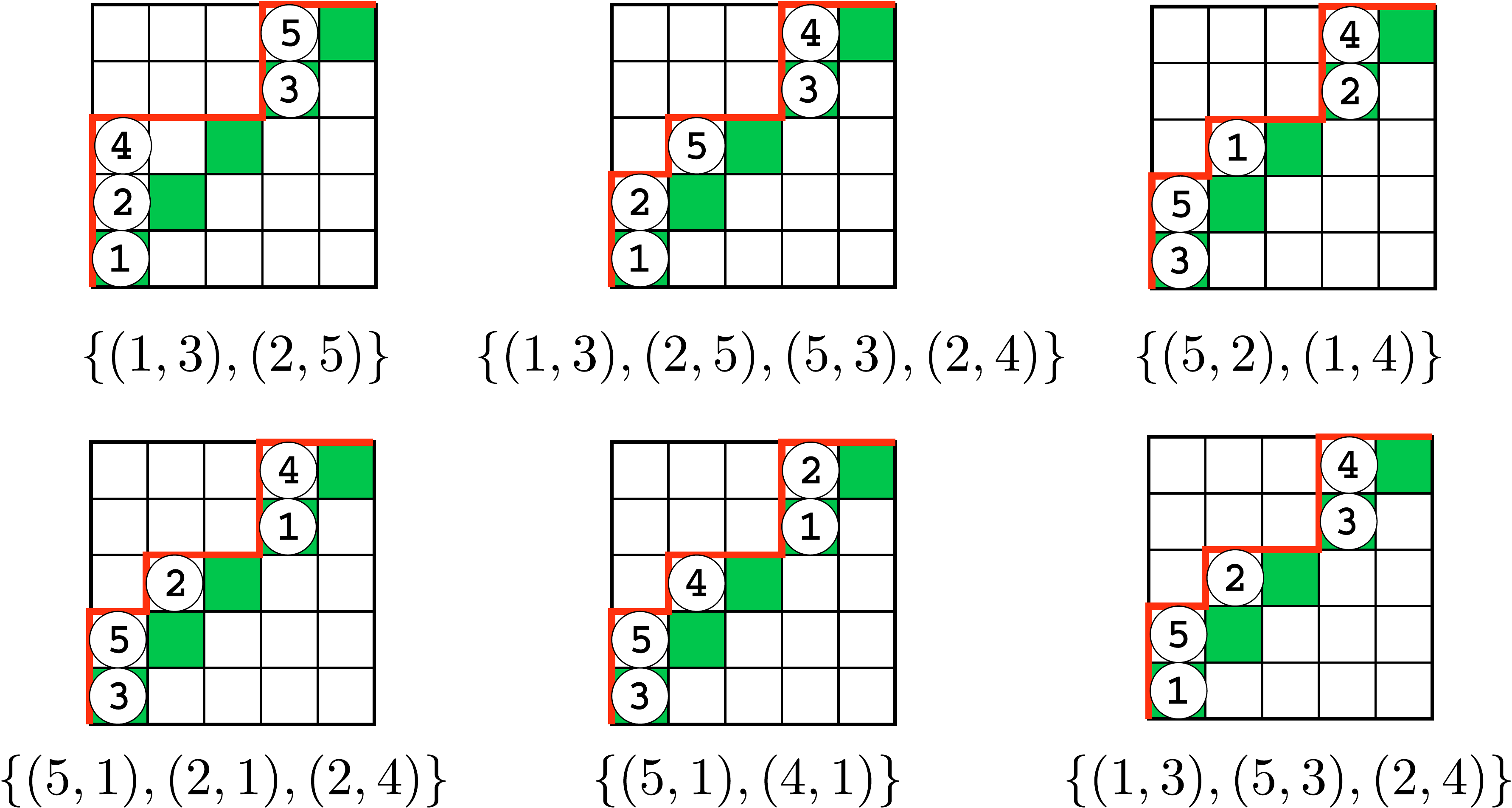}}}
$$
Thus Theorem \ref{thm:II.1} gives the equality
$$
\LL\nabla \BC_3\BC_2\, 1\scs e_1 h_2,h_2\RR\ses 
t^4q^2+t^3(q^4+2q^3+2q^2 )~.
$$

 
 Our proof of Theorem \ref{thm:II.1}, follows a similar path we used in \cite{GXZ:2011} and \cite{GXZ}. By means of  
a small collection of Macdonald polynomial identities 
(established much earlier in  \cite{BGHT:1999} and \cite{GHT:1999}) we prove a recursion satisfied by the 
left hand side of \eqref{eq:II.11}. We then show that the right hand side satisfies the same
recursion, with equality in the base cases. 

Setting, for a composition $\aa=(\aa_1,\aa_2,\ldots , \aa_\ell)$
$$
\BC_\aaa\ses \BC_{\aa_1}\BC_{\aa_2}\cdots \BC_{\aa_\ell}
$$
this  recursion,  which is the crucial result of this paper, 
may be stated as follows.
\sas

\begin{theorem}\label{thm:II.2}
Let and $a,b,c,m,n \in {\BZ}$ such that
$a+b+c = m + n$ and $\alpha \models n$.  If $m> 1$, then
\begin{align}
\LL \nabla \BC_m \BC_\alpha 1,   e_a h_bh_{c} \RR \, =\, 
 t^{m-1} &q^{\ell(\alpha)} \sum_{\beta \models m-1} 
\LL \nabla \BC_ \alpha\BC_\beta 1,  e_{a-1} h_{b } h_{c }  \RR\nonumber\\
 &+  t^{m-1} q^{\ell(\alpha)} \sum_{\beta \models m-2}    
\LL\nabla \BC_ \alpha\BC_\beta 1,  e_{a}  h_{b-1} h_{c-1}  \RR.
\label{eq:II.17}
\end{align}
If $m=1$, then 
\begin{align}
\LL \nabla \BC_1\BC_\alpha , e_a h_b h_c \RR = 
q^{\ell(\alpha)}  
\LL \nabla \BC_{\alpha} 1,  & e_{a-1} h_b h_{c}\RR +\LL \nabla \BC_{\alpha} 1,e_a h_{b-1} h_{c}+e_{a} h_b h_{c-1}\RR\nonumber\\
& + (q-1)\sum_{i:\alpha_i = 1} q^{i-1} \LL \nabla \BC_{{\widehat \alpha}^{(i)} }1, e_a h_{b-1} h_{c-1}\RR
\label{eq:II.18}
\end{align}
where  ${\widehat \alpha}^{(i)} = 
(\alpha_1, \alpha_2, \ldots, \alpha_{i-1}, \alpha_{i+1}, \ldots, \alpha_{\ell(\alpha)})$.
\end{theorem}

What is  different in 
this case in contrast with the developments in \cite{GH:2002}, \cite{GXZ:2011} and \cite{Haglund:2004},  is that there
the combinatorial side suggested the recursion that both sides had to satisfy.
By contrast in the present case we could not have even remotely come up with
the combinatorics that unravels out of the above recursion. 
In fact we shall see that \eqref{eq:II.18} will guide us,  
in a   totally unexpected manner, to carry out some remarkable 
inclusion-exclusions sieving of parking functions to prove the right hand side of \eqref{eq:II.11} satisfies this recursion.
\sa

We must also mention that Theorem \ref{thm:II.1}  may also be viewed  as a path  
result with the same flavor as Haglund's Schr\"oder result \cite{Haglund:2004}. 
This is simply obtained by converting each of our parking functions into a lattice path with the following 5 steps
\vskip -.2in
\begin{equation}
\vcenter{\hbox{\includegraphics[width=2.4in]{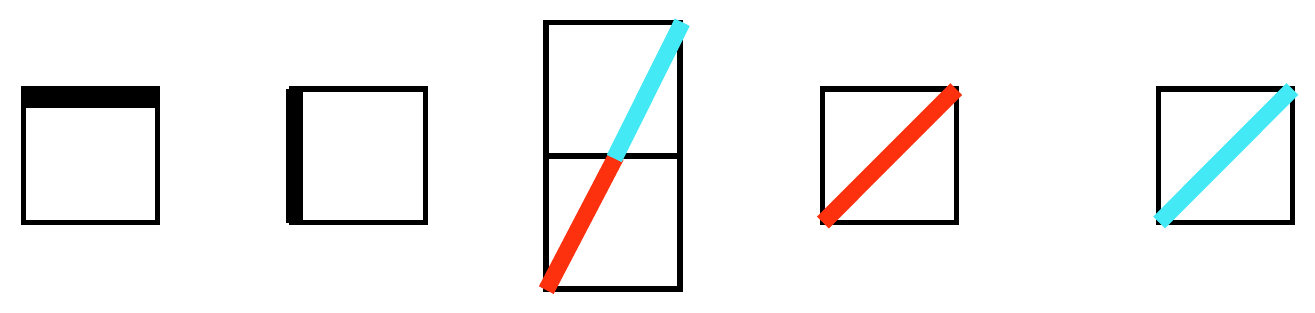}}}
\label{eq:II.19}
\end{equation} 
always remaining weakly above the main diagonal and touching it according to the
corresponding composition. The coefficient $\LL \nabla \BC_\alpha 1, e_a h_b h_c \RR$ is equal to the number of paths from $(0,0)$ to
$(n,n)$ using the 5 steps shown in the figure above
with $a$ vertical black steps, $b$ red steps and $c$ blue steps.
Areas and dinvs of our parking functions can easily
be converted  into  geometric properties  of the corresponding path.
This correspondence, which is dictated by the recursion in \eqref{eq:II.17}, 
 is obtained as follows. For convenience we will  refer to cars in the words 
 $\DA E_1$, $E_2$ and $E_3$ as \asmall,  \middle~ and 
\abig, or briefly as $S,M,B$. 
This given, the path corresponding to one of our parking functions is obtained by
deforming its supporting Dyck path according to the following rules:
\sas

\begin{enumerate}
\item {\it Every East step remains an east step.}
\item {\it  Every North step adjacent to an $S$  remains an North step.}
\item {\it  Every pair of successive North steps adjacent to a pair  $\begin{array}{c}B\\ M
\end{array}$ 
is replaced by the slope 2 step in \eqref{eq:II.19}.}
\item {\it The remaining North steps are replaced by the two slope 1 steps
in \eqref{eq:II.19}, the red  one  if adjacent to an $M$ and the blue  one if adjacent to a  $B$.}
\end{enumerate}

\noindent
In our example above the only $S$ is $1$ the $M's$ are $2,3$ and the $B's$ are $4,5$. Using the above rules these  parking functions   convert into the following six paths.
$$
\vcenter{\hbox{\includegraphics[width=5in]{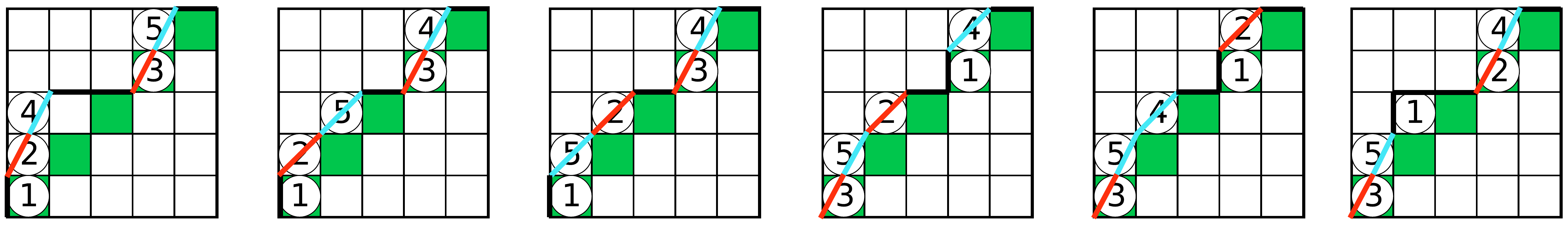}}}
$$

We have divided our presentation into three sections. In the first section we list the Macdonald polynomials identities  we plan to  use, referring to previous publications for their proofs. Next we give  proofs of some identities specifically  derived for our  present needs. This section also includes an outline of the symmetric function 
manipulations we plan to use to prove Theorem \ref{thm:II.2}. 

The second section is totally devoted to the proof of Theorem \ref{thm:II.2}. This is the most technical part of the paper, and perhaps the most suggestive of directions for  future work in this subject.
We tried whenever possible to motivate some  of the steps but in ultimate analysis 
the majority of them was forced on us by the complexities of the problem.

In the third and final section we prove that the combinatorial side satisfies the same recursion and verify  the  identities that cover the base cases. This section makes lighter reading, with only prerequisite  the statements of Theorems \ref{thm:II.1} and  \ref{thm:II.2}.
 It can be read  before embarking in \hbox{section 2.}

\sas
\noindent
{\bf Acknowledgement}: The authors would like to acknowledge the contributions and helpful guidance of
Angela Hicks on the combinatorial part of this work.

\end{section}

\begin{section}{Auxiliary  identities and results.}

In this section, after a few necessary definitions,  we  will list the identities  and prove a few preliminary results that are needed in our arguments.
The identities have been stated and proved elsewhere  but we include them here as a sequence of propositions without proofs
for the benefit of  the reader.

The space of symmetric polynomials will be denoted $\Lambda$. 
    The subspace of
homogeneous symmetric polynomials of degree $d$ will be denoted by
$\Lambda^{=d}$.  We must refer to Macdonald's exposition on symmetric functions \cite{Mac:1995} for the basic identities
that we will use here.
We will seldom work with symmetric polynomials expressed in terms of variables but rather express them
in terms of one of the classical symmetric function bases:
{\it power}    $\{p_\mu\}_\mu$, {\it monomial}  $\{m_\mu\}_\mu$, {\it homogeneous}  $\{h_\mu\}_\mu$,
{\it elementary}  $\{e_\mu\}_\mu$, and {\it Schur}  $\{s_\mu\}_\mu$.

Let us  recall that the fundamental involution $\om$ may be defined by setting 
for the power basis indexed by  $\mu=(\mu_1,\mu_2,\ldots ,\mu_k)\part n$
\begin{equation}
\om p_\mu\ses (-1)^{n-k}p_\mu\ses  (-1)^{|\mu|-l(\mu) }p_\mu
\label{eq:2.1}
\end{equation}
where for any vector $\ess v=(v_1,v_2,\cdots,v_k)$ we set
 $
\ess |v|=\sum_{i=1}^k v_i 
$
and 
$ l(v)=k $.

In dealing with symmetric function identities, specially with those arising
in the theory of Macdonald polynomials, we find it convenient
and often indispensable to use 
plethystic notation. This device has a straightforward definition. We simply set for
any expression $E=E(t_1,t_2 ,\ldots )$ and any power symmetric function $p_k$,
\begin{equation}
p_k[E]\ses E( t_1^k,t_2^k,\ldots ).
\label{eq:2.2}
\end{equation}

In particular, if $E$ is the sum of a set of variables $E = X = x_1 + x_2 + x_3 + \cdots$,
then 
$$p_k[X] = x_1^k + x_2^k + x_3^k + \cdots$$
is the power sum symmetric function evaluated at those variables.

This given, for any symmetric function $F$ we set
\begin{equation}
F[E]\ses Q_F(p_1,p_2, \ldots )\Big|_{p_k\RA E(t_1^k,t_2^k,\ldots )}
\label{eq:2.3}
\end{equation}

where $Q_F$ is the polynomial yielding the expansion of $F$ in 
terms of the power basis.

In this notation, $-E$ does not, as one might expect, correspond to replacing each
variable in an  expression $E$ with the negative of that variable.
Instead, we see that \eqref{eq:2.3}  gives
\begin{equation}
p_k[-E]\ses -p_k[E].
\label{eq:2.4}
\end{equation}

However, we will still need to carry out ordinary changes 
of signs of  the variables, and we will achieve this
by the introduction of an additional symbol $\eee$ that this to be acted upon  just like any other variable which
however, outside of the plethystic bracket, is simply replaced by $-1$. 
For instance, these conventions give for $X_k=x_1+x_2+\cdots +x_n$,

\begin{equation}
p_k[ -\eee X_n]\ses  -\eee^k\sum_{i=1}^nx_i^k\ses  (-1)^{k-1} \sum_{i=1}^nx_i^k~. 
\label{eq:2.5}
\end{equation}  

As a consequence, for any symmetric function $F\in \Lambda$ and any expression $E$ we have
$
\om F[E] \ses F[-\eee E]~.$ 
In particular,  if $F\in \Lambda^{=k} $  we may also write  
\begin{equation}
F[-E]\ses  \eee^{-k }F[-\eee E]\ses (-1)^k \om F[ E].
\label{eq:2.7}
\end{equation}

We must also mention that the formal power series
\begin{equation}
\OM\ses exp\left(\sum_{k\ge 1}\frac{p_k}{k}\right)
\label{eq:2.8}
\end{equation}
combined with plethysic substitutions will  provide   
a powerful way of dealing with the many generating functions occurring in
our manipulations.

\noindent
\begin{minipage}{0.68 \textwidth}
\hskip .2in Here and after it will be convenient to 
identify partitions with their (French) Ferrers diagram. Given
a partition $\mu$ and a cell $c\in \mu$, Macdonald introduces four parameters
$l =l_\mu(c)$, $l'=l'_\mu(c)$, $a =a_\mu(c)$ and  $a'=a'_\mu(c)$ 
called 
{\it leg, coleg, arm } and {\it coarm } which  give the 
number of lattice cells of $\mu$  strictly  {\rm NORTH},  
{\rm SOUTH}, {\rm EAST } and  {\rm WEST } of $c$, (see adjacent figure).

Following Macdonald we will set
\begin{equation}
\bigsp n(\mu)\ses \sum_{c\in\mu} l_\mu(c)\ses \sum_{c\in\mu} l'_\mu(c)
\ses \sum_{i=1}^{l(\mu)} (i-1)\mu_i. 
\label{eq:2.9}
\end{equation} 
\end{minipage}
$
\vcenter{\hbox{\hskip .1in\includegraphics[width=1.8in]{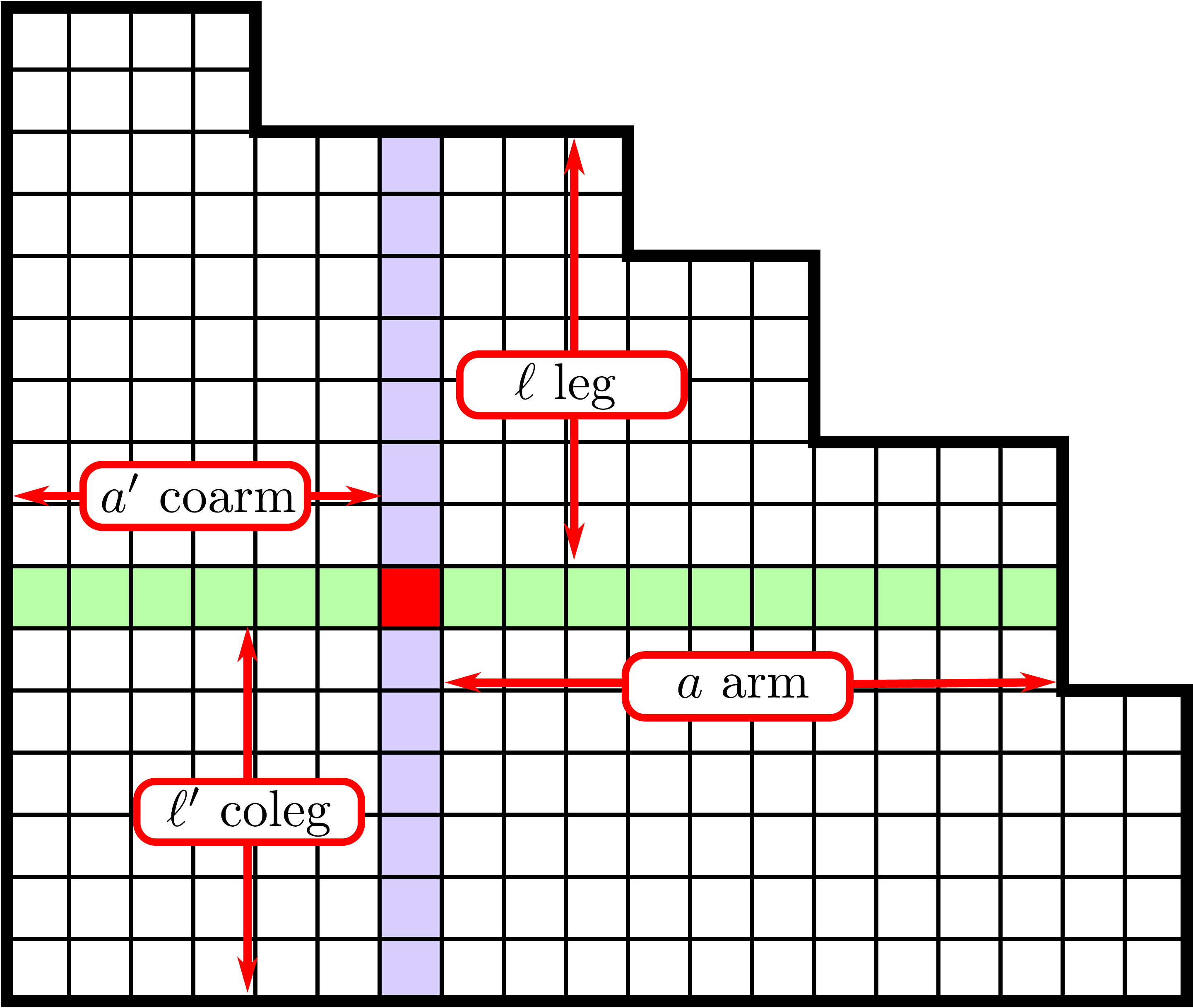}}}
$

Denoting by $\mu'$ the conjugate of $\mu$, 
the notational ingredients playing a role in the theory of Macdonald
polynomials are  

\begin{align}
T_\mu=t^{n(\mu)}q^{n(\mu')}~&,
\ess\ess B_\mu(q,t)  =\sum_{c\in \mu}t^{l'_\mu(c)} q^{a'_\mu(c)}~,\cr
\Pi_\mu(q,t)= \prod_{c\in \mu;c\neq(0,0)}(1-&t^{l'_\mu(c)}q^{a'_\mu(c)})~,\ess\ess M=(1-t)(1-q)~,\label{eq:2.10}\\
D_\mu(q,t)= MB_\mu(q,t)- 1 ~,\ess\ess&
 w_\mu(q,t)=\prod_{c\in \mu}(q^{a _\mu(c)} -t^{l _\mu(c)+1})(t^{l _\mu(c)} -q^{a _\mu(c)+1}). \nonumber
\end{align}
We will also use  a deformation of the Hall scalar product, which we call
the {\it star} scalar product,   defined by setting for the power basis 
$$
\LL p_\la\scs p_\mu \RR_*\ses 
(-1)^{|\mu|-l(\mu)} \prod_i (1-t^{\mu_i})(1-q^{\mu_i})\ssp z_\mu\ssp
\chi(\la=\mu) ,
$$
where $z_\mu$ gives the order of the stabilizer of a permutation with cycle structure $\mu$.
\sa

This given, the modified Macdonald Polynomials we will deal with here are the unique 
symmetric function basis $\big\{\TH_\mu(X;q,t)\big \}_\mu$ which  is
upper triangularly related to the basis $\{s_\la[\frac{X}{t-1}]\}_\la$
and satisfies the orthogonality condition
\begin{equation}
\LL \TH_\la\scs \TH_\mu\RR_*\ses \chi(\la=\mu) w_\mu(q,t)~.
\label{eq:2.11}
\end{equation}

In this writing we will make intensive use of the operator $\nabla$ 
defined by setting for all partitions $\mu$
$$
\nabla \TH_\mu \ses T_\mu \TH_\mu.
$$

The following identities will play a crucial role in our present developments. 
Their proofs can be found in \cite{BGHT:1999}, \cite{GH:1996} and \cite{GHT:1999}.


\begin{prop}(Macdonald's Reproducing Kernel)\label{prop:2.1}
The orthogonality relations in \eqref{eq:2.11} yield the Cauchy identity
for our Macdonald polynomials in the form
$$
\OM\left[ \frac{-\eee XY}{M} \right]\ses \sum_{\mu}
\frac{\TH_\mu[X]\TH_\mu[Y]}{w_\mu}
$$
which restricted to its homogeneous component of degree $n$ in $X$ and $Y$ 
reduces to
\begin{equation}
e_n\left[  \frac{XY}{M} \right]\ses \sum_{\mu\part n}
\frac{\TH_\mu[X]\TH_\mu[Y]}{w_\mu}~.
\label{eq:2.12}
\end{equation}
\end{prop}


\begin{prop}\label{prop:2.2}
For all pairs of partitions $\aaa,\bbb$ we have
\begin{equation}
a)\ess\ess\frac{\TH_\aaa[M B_\bbb ] }{  \Pi_\aa}
\ses 
\frac{\TH_\bbb[ M B_\aaa ]}{\Pi_\bb}
\scs \bigsp
b)\ess\ess(-1)^{|\aaa|}\frac{\TH_\aaa[D_\bbb ]}{T_\aaa}
\ses 
(-1)^{|\bbb|}\frac{\TH_\bbb[ D_\aaa ]}{T_\bb}
\label{eq:2.13}
\end{equation}
\end{prop}
   
Let us also recall that the  coefficients  $c_{\mu\nu}$ and $d_{\mu\nu}$ occurring in the Macdonald Pieri formulas
\begin{equation}
a)\ess\ess e_1\TH_\nu\ses \sum_{\mu\LA \nu}d_{\mu\nu}\TH_\mu\scs
\ess\ess\ess\ess\ess\ess\ess\ess\ess
b)\ess\ess e_1^\perp\TH_\mu\ses \sum_{\nu\RA \mu}c_{\mu\nu}\TH_\nu\scs
\label{eq:2.15}
\end{equation}
are related by the identity
\begin{equation}
d_{\mu\nu}\ses M c_{\mu\nu} \frac{w_\nu}{w_\mu} ~.
\label{eq:2.16}
\end{equation}
This given, the following summations formulas proved in \cite{GH:1998} and \cite{Zab:2005} are also indispensable here.
\sas
\sas

\begin{prop}\label{prop:2.3}
For $\mu$ a partition of a positive integer,
\begin{equation}
\sum_{\nu\RA\mu}c_{\mu\nu}(q,t)\, (T_\mu/T_\nu)^k\ses 
\begin{cases}
\frac{tq}{M}\ssp h_{k+1}\big[D_\mu(q,t)/tq\big] &\hbox{ if }k\geq 1\\
B_\mu(q,t) &\hbox{ if }k=0
\end{cases}
\label{eq:2.17}
\end{equation}
For $\nu$ a partition of a non-negative integer,
\def \mur {\leftarrow}
\begin{equation}
\sum_{\mu\mur\nu}d_{\mu\nu}(q,t)\, (T_\mu/T_\nu)^k = 
\begin{cases}
(-1)^{k-1}\ssp e_{k-1}\big[D_\nu(q,t) \big] &\hbox{ if } k\geq 1 \\
1 &\hbox{ if }k=0 .
\end{cases}
\label{eq:2.18}
\end{equation}
Here $\nu\RA \mu$ simply means that the sum is over  $\nu$'s   obtained 
from $\mu$ by removing a corner cell and  $\mu\LA \nu$ 
means  that the sum is over $\mu$'s   obtained 
from $\nu$ by adding a corner cell.
\end{prop}

Recall that the Hall scalar product in the theory of Symmetric functions
may be defined by setting, for the power basis
\begin{equation}
\LL p_\la\scs p_\mu \RR \ses 
 z_\mu\ssp
\chi(\la=\mu) ~.
\label{eq:2.19}
\end{equation}
It follows from this that
the $*$-scalar product, is simply related to the   
Hall scalar product
by setting for all pairs of symmetric functions $f,g$,
\begin{equation}
\LL f\scs g\RR_* \ses \LL f\scs \om \phi g\RR,
\label{eq:2.20}
\end{equation}
where it has been customary to let $\phi$ be the operator
defined by setting for any symmetric function $f$
\begin{equation}
\phi\, f[X]\ses f[MX].
\label{eq:2.21}
\end{equation}
Note that the inverse of $\phi$ is usually written in the form
\begin{equation}
f^*[X]\ses f[X/M]~.
\label{eq:2.22}
\end{equation}
In particular we also have  for all symmetric functions $f,g$
\begin{equation}
\LL f\scs g\RR \ses \LL f, \om g^* \RR_* 
\label{eq:2.23}
\end{equation}

Note that the orthogonality relations in \eqref{eq:2.11} yield us 
the following  Macdonald polynomial expansions 

\begin{prop}\label{prop:2.4}
For all $n\ge 1$, we have
\begin{align}
&a)\ess\ess  
e_n\big[\tttt{}\frac{X}{M} \big]= \sum_{\mu\part n} \frac{\TH_\mu[X]}{w_\mu}~,\\
&b)\ess \ess 
h_n\big[\tttt{}\frac{X}{M} \big]= \sum_{\mu\part n} \frac{T_\mu \TH_\mu[X]}{w_\mu}~,\\
&c)\ess 
h_k\big[\tttt{}\frac{X}{M} \big] e_{n-k}\big[\tttt{}\frac{X}{M} \big] 
= \sum_{\mu\part n} \frac{e_k[B_\mu] \TH_\mu[X]}{w_\mu}~.
\end{align}
\end{prop}

The following identity (proved in \cite{GHT:1999}) may sometimes provide an explicit  expression for the scalar product
of a Macdonald polynomial  with a symmetric polynomial which contains an $h$ factor.
More precisely, we have the following proposition.


\begin{prop}\label{prop:2.5}
 For all $f\in \Lambda^{= r}$ and $\mu\vdash n$ we have
$$
\langle f h_{n-r}\, , \, \tilde H_\mu\rangle \, =\, \nabla^{-1}\big(\omega 
f\left[\frac{X-\epsilon }{ M }\right]  \big)
\big|_{X\rightarrow MB_\mu-1}~.
$$
\end{prop}

For instance  for $f=e_r$ this identity combined with a) and  b) of Proposition \ref{prop:2.4} gives
\begin{align}
\LL \TH_\mu\scs e_rh_{n-r}\RR
&=
\nabla^{-1}h_r\left[\frac{X-\eee}{M}\right]\Big|_{X\RA MB_\mu-1}
= \sum_{i=0}^re_{r-i}\left[\frac{1}{M}\right]
\nabla^{-1}h_i\left[\frac{X}{M}\right]\Big|_{X\RA MB_\mu-1}
\cr
 &= \sum_{i=0}^re_{r-i}\left[\frac{1}{M}\right]
\sum_{\ggg\part i } \frac{ \TH_\ggg [MB_\mu-1]}{w_\ggg}
=
\sum_{i=0}^re_{r-i}\left[\frac{1}{M}\right]
e_i\left[\frac{  {MB_\mu-1}}{ M}  \right] 
\ses e_r[B_\mu]~.
\end{align}
We thus have
\begin{equation}
\LL  \TH_\mu\scs e_rh_{n-r}\RR\ses e_r[B_\mu]~.
\label{eq:2.24}
\end{equation}
This is an identity  which will provide an important step in the proof of Theorem \ref{thm:II.2}.
\sas

In addition to the $\BC_a$ operator given in \eqref{eq:II.9} which we recall acts on a symmetric polynomial according to the plethystic formula
\begin{equation}  
 \BC_a P[X]\ses \left(-\frac{1}{q}\right)^{a-1} P\left[X-\frac{1-1/q }{ z}
\right]\displaystyle
\sum_{m\ge 0}z^m h_{m}[X]\, \Big|_{z^a},
\label{eq:2.25}
\end{equation}
Haglund-Morse-Zabrocki in  \cite{HMZ:2012} introduce also the $\BB_a$ operator obtained by setting
\begin{equation}
 \BB_a  P[X]\ses  P\left[  X+\eee  \frac{1-q }{ z}\right]
\sum_{m\ge 0}\,z^m e_m[X] \Big|_{z^a}~.
\label{eq:2.26}
\end{equation}
They showed there that these operators, for $a+b>0$ have the commutativity relation
\begin{equation}
\BB_a\BC_b\ses q \, \BC_b \BB_a~.
\label{eq:2.27}
\end{equation}
  However,  we will 
need here the  more refined identity that  yields the interaction of the  $\BB_{-1}$
and $\BC_1$ operators.  The  following identity  also 
explains what happens  when $a+b< 0$.

\begin{prop}\label{prop:2.6}
For all pairs of integers $a,b$   we have
\begin{equation}
 \ess\big( q\,  \BC_b \BB_a - \BB_a\TBH_b\big) P[X]
\, =\, 
 (q-1)(-1)^{a+b-1}/q^{b-1} \times \begin{cases}
0 &\hbox{ if }a+b>0\\
P[X]&\hbox{ if }a+b=0\\
P [X + (\tttt{}\frac{1}{q}-q)/z ] \Big|_{z^{a+b}}~.
&\hbox{ if }a+b<0~.
\end{cases}
\label{eq:2.28}
\end{equation}
\end{prop}

\begin{proof}
Using \eqref{eq:2.25} we get for $P\in \La^{=d}$,
\begin{align}
(-q)^{b-1}\TBH_b P[X]
&= \sum_{r_1=0}^dP \left[X-\frac{1-1/q }{ z_1}\right]\Big|_{z_1^{-r_1}}
\sum_{m\ge 0}z ^mh_m[X]\, \Big|_{z^{b+r_1}} \\
&\ses  
\sum_{r_1=0}^dP\left[X-\frac{1-1/q }{z_1}\right]\Big|_{z_1^{-r_1}}h_{b+r_1}[X] ~,
\end{align}
Now it will be convenient go set  for symmetric polynomial $P[X]$
\begin{equation}
P^{(r_1,r_2)}[X]\ses 
P\left[X  -\frac{1-1/q}{z_1}+\eee  \frac{1-q }{ z_2}\right]
\Big|_{z_1^{-r_1}z_2^{-r_2}}
\end{equation}
This given, \eqref{eq:2.26}  gives
\begin{align}
(-q)^{b-1}\BB_a\TBH_b P[X]\nonumber
&\ses
\sum_{r_1=0}^dP \left[ X-\frac{1-1/q }{ z_1}+\eee \frac{1-q }{ z_2}\right]\Big|_{z_1^{-r_1} z_2^{-r_2}} h_{b+r_1}
\left[ X+\eee \frac{1-q}{z_2}\right]\sum_{m\ge 0}\,z_2^m e_m[X] \, \Big|_{z_2^{a+r_2}}\\
&\ses
\sum_{r_1,r_2=0}^dP^{(r_1,r_2)}[X]  \sum_{s=0}^{b+r_1}h_{b+r_1-s}[X](-1)^s
h_s\big[  1-q \big]\sum_{m\ge 0}\,z_2^m e_m[X] \, \Big|_{z_2^{a+r_2+s}}\\
&\ses
\sum_{r_1,r_2=0}^d\sum_{s=0}^{b+r_1}
P^{(r_1,r_2)}[X] 
h_{b+r_1-s}[X](-1)^sh_s\big[  1-q\big] e_{a+r_2+s}[ X]\, . 
\end{align}

 Since it is easily shown that 
\begin{equation}
 h_s\big[  1-q\big] = 
\begin{cases}
  1 &\hbox{ if }s=0 \\
  1-q &\hbox{ if }s>0~,
\end{cases}
\label{eq:2.29}
\end{equation}
 we can  write
\begin{align}
 (-q)^{b-1}\BB_a\TBH_b P[X]
 &\ses 
 \sum_{r_1,r_2=0}^d 
 P^{(r_1,r_2)}[X] 
 h_{b+r_1 }[X]  e_{a+r_2 }[   X]\sps \\
 &\sps 
 (1-q)\sum_{r_1,r_2=0}^d\sum_{s=1}^{b+r_1}
 P^{(r_1,r_2)}[X] 
 h_{b+r_1-s}[X](-1)^s  e_{a+r_2+s}[   X]
\end{align}
 and the change of summation index $u=a+r_2+s$ gives
\begin{align}
\nonumber (-q)^{b-1}\BB_a\TBH_b P[X]
 &\ses 
 \sum_{r_1,r_2=0}^d 
 P^{(r_1,r_2)}[X] 
 h_{b+r_1 }[X]  e_{a+r_2 }[   X]\sps \\
 &\ess\ess\ess\ess\ess\sps 
 (-1)^a(1-q)\sum_{r_1,r_2=0}^d\sum_{u=a+r_2+1}^{a+b+r_1+r_2}
 P^{(r_1,r_2)}[X] 
 h_{a+b+r_1+r_2- u }[X](-1)^{u -r_2} e_u[   X]~.
\end{align}

 \noindent
 An entirely similar manipulation  gives 
\begin{align}
 (-q)^{b-1}\ess q\, \BC_b \BB_aP[X]
 &\ses \nonumber
 \sum_{r_1,r_2=0}^dP^{(r_1,r_2)} [X ]  
 e_{r_2+a }[ X]\,  
 h_{r_1+b}[X]  \\
 &\ess\ess  
 \sms(-1)^a\big( 1-q    \big)
 \sum_{r_1,r_2=0}^d P^{(r_1,r_2)} [X ]  
 \sum_{u=0}^{a+r_2} 
 e_u[  X]\,(-1)^{u+r_2  } \, 
 h_{a+b+r_1+r_2-u}[X]~,
\end{align}
 and thus by subtraction we get
\begin{align}
\nonumber (-q)^{b-1}\ess\big( &q\, \BC_b \BB_aP[X]\sms \BB_a\TBH_b\big) P[X]\\
 &\ses
  \big(         q-1    \big)(-1)^a
 \sum_{r_1,r_2=0}^d P^{(r_1,r_2)} [X ]  
 \sum_{u=0}^{a+b+r_1+r_2} 
 h_u[-  X]\,(-1)^{r_2  } \, 
 h_{a+b+r_1+r_2-u}[X]\\
&\ses
  \big(         q-1    \big)(-1)^a
 \sum_{r_1,r_2=0}^d P^{(r_1,r_2)} [X ]  (-1)^{r_2  } \,   
 h_{a+b+r_1+r_2 }[X-X]~.
\end{align}
 Since the homogeneous symmetric function $h_a$ on an empty alphabet vanishes  unless $a=0$, using the definition of $P^{(r_1,r_2)} [X ]$, this reduces to
\begin{equation}
  \big( q\, \BC_b \BB_aP[X]\sms \BB_a\TBH_b\big) P[X]
\,=\,  \frac{(q-1)(-1)^{a+b-1} }{ q^{b-1}}\hskip -.2in
\sum_{\multi{r_1,r_2= 0\cr
r_1+r_2=-(a+b)}} ^{d}\hskip -.2in
(-1)^{r_2}
 P[X  -\tttt{1-1/q \over z_1}+\tttt{\eee  \tttt{1-q \over z_2}}]
 \Big|_{
 {1\over z_1^{r_1} }
 {1\over z_2^{ r_2}  }
 }
\label{eq:2.30}
\end{equation}
which immediately gives the first two cases of \eqref{eq:2.28}. 
For the remaining case, notice that for suitable coefficients $a_{r_1,r_2}$
we may write 
$$
 P[X  -\tttt{1-1/q \over z_1}+\tttt{\eee  \tttt{1-q \over z_2}}]
 \Big|_{
 {1\over z_1^{r_1} }
 {1\over z_2^{ r_2}  }}\ses \eee^{r_2}a_{r_1,r_2}
 $$
and the sum on the right hand side of \eqref{eq:2.30} may be rewritten as 
$$
\sum_{\multi{r_1,r_2= 0\cr
r_1+r_2=-(a+b)}} ^{d}\hskip -.2in
(-1)^{r_2} \eee^{r_2}a_{r_1,r_2}
\ses\hskip -.2in
\sum_{\multi{r_1,r_2= 0\cr
r_1+r_2=-(a+b)}} ^{d}\hskip -.2in
a_{r_1,r_2}
\ses\hskip -.2in
\sum_{\multi{r_1,r_2= 0 
 }} ^{d }
a_{r_1,r_2}
\tttt{ {1\over z ^{r_1} }
 {1\over z ^{ r_2}  }}\Big|_{1\over z^{-(a+b)}}
 \ses 
 P[X  -\tttt{1-1/q \over z }+\tttt{   \tttt{1-q \over z }}]
 \Big|_{z^{a+b}},
$$
which is precisely as asserted in \eqref{eq:2.28}.
\end{proof} 

We are now finally in a position to give an outline of the manipulations that we will carry out to prove the recursion of Theorem \ref{thm:II.2}.
\sas

A  close examination of our previous work, most particularly our result in \cite{GXZ:2011} and \cite{GXZ}, 
might suggest that  also in the present case  we should also have
(for $m>1$)  an  identity of the form
\begin{equation}
\LL \nabla \BC_m \BC_\alpha 1,   e_a h_bh_{c} \RR \, =\, 
 t^{m-1} q^{\ell(\alpha)} \sum_{\beta \models m-1}  
\LL \nabla \BC_\alpha\BC_\beta 1,  e_{a-1} h_{b } h_{c }\RR~.
\label{eq:2.31}
\end{equation}
Now we can easily see from Theorem \ref{thm:II.2} that an additional  term  is need  to obtain a true identity. But to give an idea of how one may 
end up discovering the  correct recursion, let us pretend that \eqref{eq:2.31} is true and try to derive from it what  is needed to prove it.

Note first that using  \eqref{eq:II.12}  for $n=m-1$ and \eqref{eq:2.26} with $P[X]= 1$ we obtain that 
$$
\sum_{\bbb\models m-1}\BC_\bbb\, 1 \ses e_{m-1}[X]\ses \BB_{m-1}\, 1~.
$$
This allows us to rewrite \eqref{eq:2.31} in the more compact form 
$$
\LL \nabla \BC_m \BC_\beta 1,   e_a h_bh_{c} \RR \, =\, 
 t^{m-1} q^{\ell(\alpha)}  
\LL \nabla \BC_ \alpha\ \BB_{m-1}1,  e_{a-1} h_{b } h_{c }\RR
$$
and,  by a multiple use of the commutativity relation in \eqref{eq:2.27}, we can further simplify  this to
$$
\LL \nabla \BC_m \BC_\alpha 1,   e_a h_bh_{c} \RR \, =\, 
 t^{m-1}    
\LL \nabla \BB_{m-1} \BC_ \alpha\ 1,  e_{a-1} h_{b } h_{c }\RR~.
$$
Following the manipulations carried out in \cite{GXZ:2011} our next step is to pass to $*$-scalar products (using \eqref{eq:2.23})  and obtain
$$
\LL \nabla \BC_m \BC_\alpha 1,   h_a^* e_b^*e_{c}^* \RR_* \, =\, 
 t^{m-1}    
\LL \nabla \BB_{m-1} \BC_ \alpha\ 1,  h_{a-1}^* e_{b }^* e_{c }^*\RR_*~.
$$
Since the $\nabla$ operator is self adjoint with respect to the $*$-scalar product, this in turn can be rewritten as
\begin{equation}
\LL    \BC_\alpha 1   \scs \BC_m^* \nabla\,  h_a^* e_b^*e_{c}^* \RR_* \, =\, 
 t^{m-1}    
\LL    \BC_ \alpha\, 1  \scs \BB_{m-1} ^* \nabla\,  h_{a-1}^* e_{b }^* e_{c }^*\RR_*\scs
\label{eq:2.32}
\end{equation}
 where $\BC_m^*$ and $\BB_{m-1} ^* $ denote the $*$-scalar product  adjoints of  $\BC_m $ and $\BB_{m-1}  $.

Now,  except  for a minor rescaling, the    $ \BC_ \alpha\,1$ are essentially   Hall-Littlewood polynomials
and the latter, for $\aaa$  a partition,  are a well known basis. Thus  \eqref{eq:2.32} can be true if and only if 
we have the symmetric function identity 
\begin{equation}
\BC_m^* \nabla\,  h_a^* e_b^*e_{c}^*\ses t^{m-1}  \BB_{m-1} ^* \nabla\,  h_{a-1}^* e_{b }^* e_{c }^*~.
\label{eq:2.33}
\end{equation}

Starting from the left hand side of \eqref{eq:2.33}, a remarkable sequence of manipulations, step by step guided by our previous work, reveals that
\eqref{eq:2.33} is but a tip of an iceberg. In fact, as we will show in the next section that the correct form of \eqref{eq:2.33} may be stated as follows.
\sas

\begin{theorem}\label{thm:2.1}
Let and $a,b,c,m,n \in {\BZ}$ be such that
$a+b+c = m + n$ and $\alpha \models n$.  Then
\begin{align}
 \BC_m^* \nabla\,  h_a^* e_b^*e_{c}^*\ses    t^{m-1}   
 \BB_{m-1} ^* \nabla\, &  h_{a-1}^* e_{b }^* e_{c }^*\sps  t^{m-1}  \BB_{m-2} ^* \nabla\,  h_{a}^* e_{b-1 }^* e_{c-1 }^*\sps \nonumber\\
&\ess\ess\ess\ess
\sps \chi(m=1)\big( \nabla   h_{a}^* e_{b-1 }^* e_{c }^*\sps   \nabla h_{a}^* e_{b }^* e_{c-1 }^*\big)~.
\label{eq:2.34}
\end{align}
\end{theorem}
 
Postponing to next section the proof of this result, we will terminate this section by showing the following connection with the results
stated in the introduction.
\sas

\begin{prop}\label{prop:2.7}
Theorems \ref{thm:II.2} and \ref{thm:2.1} are equivalent.
\end{prop}

\begin{proof}
Note first  that for $m\ge 2$ the identity in \eqref{eq:2.34} reduces to 
\begin{equation}
\BC_m^* \nabla\,  h_a^* e_b^*e_{c}^*\ses    t^{m-1}   \BB_{m-1} ^* \nabla\,    h_{a-1}^* e_{b }^* e_{c }^*\sps  t^{m-1}  \BB_{m-2} ^* \nabla\,  h_{a}^* e_{b-1 }^* e_{c-1 }^*~,
\label{eq:2.35}
\end{equation}
and the recursion in \eqref{eq:II.17} is simply obtained by reversing the steps that got us from \eqref{eq:2.31} to \eqref{eq:2.33}. We will carry them out for sake of completeness.
 To begin we take the $*$-scalar product with by $\BC_\aaa\, 1$ on both sides of this equation to obtain
 $$
\LL \BC_\aaa\, 1 \scs \BC_m^* \nabla\,  h_a^* e_b^*e_{c}^*\RR_*\ses    t^{m-1}  
\LL \BC_\aaa\, 1  \scs  \BB_{m-1} ^* \nabla\,    h_{a-1}^* e_{b }^* e_{c }^*\RR_*\sps  t^{m-1} \LL \BC_\aaa\, 1  
\scs  \BB_{m-2} ^* \nabla\,  h_{a}^* e_{b-1 }^* e_{c-1 }^*\RR_*
$$
 and moving  all the operators from  the right to the left of the $*$-scalar product gives
  $$
\LL \nabla\,  \BC_m \BC_\aaa\scs   h_a^* e_b^*e_{c}^*\RR_*\ses    t^{m-1}  
\LL \nabla\,  \BB_{m-1}  \BC_\aaa\scs       h_{a-1}^* e_{b }^* e_{c }^*\RR_*\sps  t^{m-1} \LL  \nabla\,  \BB_{m-2}  \BC_\aaa
\scs    h_{a}^* e_{b-1 }^* e_{c-1 }^*\RR_*~.
$$
In terms of the ordinary Hall scalar product this identity becomes
\begin{equation}
\LL \nabla\,  \BC_m \BC_\aaa\, 1 \scs   e_a h_b h_{c} \RR \ses    t^{m-1}  
\LL \nabla\,  \BB_{m-1}  \BC_\aaa\, 1  \scs       e_{a-1}  h_{b } h_{c } \RR \sps  t^{m-1} \LL  \nabla\,  \BB_{m-2}  \BC_\aaa\, 1  
\scs    e_{a}  h_{b-1 }  h_{c-1}  \RR ~.
\label{eq:2.36}
\end{equation}
The commutativity relation in \eqref{eq:2.27} then  gives 
  $$
\LL \nabla\,  \BC_m \BC_\aaa\, 1  \scs   e_a h_b h_{c} \RR \ses    t^{m-1} q^{\ell(\aaa)} 
\LL \nabla\,  \BC_\aaa  \BB_{m-1} \, 1\scs       e_{a-1}  h_{b } h_{c } \RR \sps  t^{m-1}  q^{\ell(\aaa)} 
\LL  \nabla\,    \BC_\aaa \BB_{m-2}\, 1  \scs    e_{a}  h_{b-1 }  h_{c-1}  \RR ~.
$$
\vskip -.08in
\noindent
Finally the two expansions
\vskip -.2in
$$
\BB_{m-1}\, 1\ses \sum_{\bbb\models m-1 }\BC_\bbb 1\ess\ess\ess \ess\ess\ess
\hbox{and} \ess\ess\ess \ess\ess\ess  \BB_{m-2}\, 1\ses \sum_{\bbb\models m-2 }\BC_\bbb 1
$$

\vskip -.08in
\noindent
give that
\begin{align}
\LL \nabla \BC_m \BC_\alpha 1,   e_a h_bh_{c} \RR \, =\, 
 t^{m-1} q^{\ell(\alpha)} \sum_{\beta \models m-1}& 
\LL \nabla \BC_ \alpha\BC_\beta 1,  e_{a-1} h_{b } h_{c }  \RR\\
 &+  t^{m-1} q^{\ell(\alpha)} \sum_{\beta \models m-2}    
\LL\nabla \BC_ \alpha\BC_\beta 1,  e_{a}  h_{b-1} h_{c-1}  \RR~.
\end{align}
\vskip -.13in
\noindent
This shows that, for $m>1$, \eqref{eq:2.34} implies \eqref{eq:II.17}.

Next note that for $m=1$, \eqref{eq:2.34} reduces to
$$ 
\BC_1^* \nabla\,  h_a^* e_b^*e_{c}^*\ses \BB_{0} ^* \nabla\,  h_{a-1}^* e_{b }^* e_{c }^*\sps     
\BB_{-1} ^* \nabla\,  h_{a}^* e_{b-1 }^* e_{c-1 }^*\sps 
\big( \nabla   h_{a}^* e_{b-1 }^* e_{c }^*\sps \nabla h_{a}^* e_{b }^* e_{c-1 }^*\big)
$$

and the same steps that brought us from \eqref{eq:2.35} to \eqref{eq:2.36} yield us the identity
\begin{align}
\LL\nabla  \BC_1\BC_\aaa 1\scs e_{a}h_{b}h_{c}\RR 
\ses 
\LL\nabla  \BB_0\BC_\aaa 1\scs & e_{a-1}h_{b}h_{c}\RR\sps  
\LL\nabla  \BB_{-1}\BC_\aaa 1\scs e_{a }h_{b-1}h_{c-1}\RR\sps
\nonumber\\
&\sps
\LL \nabla \BC_\aaa\, 1 \scs e_ah_{b-1}h_c\RR\sps \LL \nabla \BC_\aaa\, 1 \scs e_ah_{b}h_{c-1}\RR~.
\label{eq:2.37}
\end{align}

Here the only thing that remains to be done is moving $\BB_0$ and $\BB_{-1}$ to the right past all components of $\BC_\aaa$.
This can be achieved by the following two cases of \eqref{eq:2.28}
\begin{equation}
a)\ess \BB_0 \BC_b=q \BC_b \BB_0\bigsp\hbox{and}\bigsp 
b)\ess \BB_{-1} \BC_b=q \BC_b \BB_{-1}\sps 
\chi(b=1) \, (q-1) I
\bigsp (\hbox{for all $b\ge 1$})
\label{eq:2.38}
\end{equation}
where $I$ denotes the identity operator. Now using the first we  immediately obtain
\begin{equation}
\LL\nabla  \BB_0\BC_\aaa 1\scs   e_{a-1}h_{b}h_{c}\RR\ses q^{\ell(\aaa)} \LL \nabla \BC_\aaa \, 1\scs e_{a-1}h_{b}h_{c}\RR,
\label{eq:2.39}
\end{equation}
since \eqref{eq:2.26} gives $\BB_0\,1 =1$.

The effect of \eqref{eq:2.38} b) on the polynomial  $\LL\nabla  \BB_{-1}\BC_\aaa 1\scs e_{a }h_{b-1}h_{c-1}\RR$ is best understood
by working out an example. Say $\aaa=\bb\, 1\,\gg\, 1\,\dd\, 1\TT $ with $\bb,\gg,\dd,\TT$ compositions, 
the $1\,'s$  in positions $i_1,i_2,i_3$ and $\ell(\aa)=k$.
In this case multiple applications of \eqref{eq:2.38} b) yield the sum of the three terms on the right.
\begin{equation*}
\begin{matrix}
B_{-1}\BC_\aaa\, 1=&B_{-1}\BC_{\bb}C_1\BC_{\gg}C_1\BC_{\dd}C_1\BC_{\TT} \, 1&  & \\
&\DA & & \\
&q^{i_1-1}\BC_{\bb}B_{-1}C_1\BC_{\gg}C_1\BC_{\dd}C_1\BC_{\TT}\, 1 & \SRA &  (q-1) q^{i_1-1}\BC_{\bb}\bu \BC_{\gg}C_1\BC_{\dd}C_1\BC_{\TT}\, 1\\
&\DA & & \\
&q^{i_2-1}\BC_{\bb}C_1\BC_{\gg}B_{-1} C_1\BC_{\dd}C_1\BC_{\TT}\, 1 & \SRA & (q-1) q^{i_2-1}\BC_{\bb}C_1\BC_{\gg}\bu \BC_{\dd}C_1\BC_{\TT}\, 1\\
&\DA & & \\
&q^{i_3-1}\BC_{\bb}C_1\BC_{\gg} C_1\BC_{\dd}B_{-1} C_1\BC_{\TT} \, 1& \SRA & (q-1) q^{i_3-1}\BC_{\bb}C_1\BC_{\gg}C_1\BC_{\dd}\bu \BC_{\TT}\, 1\\
&\DA & & \\
&q^k\,\,\,  \BC_{\bb}C_1\BC_{\gg} C_1\BC_{\dd} C_1\BC_{\TT} B_{-1} \, 1&\ses 0&  \hbox{(since $\BB_{-1}\, 1=0$)}\\
\end{matrix}
\end{equation*}

and  we can easily see that in general we will have
$$
\BB_{-1}\BC_\aaa\, 1 \ses (q-1)\sum_{i:\alpha_i = 1} q^{i-1}  \BC_{{\widehat \alpha}^{(i)} }1~.
$$

\vskip-.3in
\noindent
This gives that
\begin{equation}
\LL\nabla  \BB_{-1}\BC_\aaa 1\scs e_{a }h_{b-1}h_{c-1}\RR
\ses (q-1)\sum_{i:\alpha_i = 1} q^{i-1}\LL \nabla  \BC_{{\widehat \alpha}^{(i)} }1  \scs e_{a }h_{b-1}h_{c-1}\RR
\label{eq:2.40}
\end{equation}
where, as before we have set ${\widehat \alpha}^{(i)} = 
(\alpha_1, \alpha_2, \ldots, \alpha_{i-1}, \alpha_{i+1}, \ldots, \alpha_{\ell(\alpha)})$. 

Thus combining \eqref{eq:2.37}, \eqref{eq:2.39} and \eqref{eq:2.40} we  obtain 
\begin{align*}
\LL \nabla \BC_1\BC_\alpha , e_a h_b h_c \RR \, =\, 
q^{\ell(\alpha)}  
\LL \nabla \BC_{\alpha}   1\scs  & e_{a-1} h_b \, h_{c}\RR +\LL \nabla \BC_{\alpha}   1,e_a h_{b-1}\,  h_{c}+e_{a} h_b  \, h_{c-1}\RR\\
& + 
(q-1)\sum_{i:\alpha_i = 1} q^{i-1} \LL \nabla \BC_{{\widehat \alpha}^{(i)} }1\scs  
 e_a h_{b-1} h_{c-1}\RR
\end{align*}
proving that \eqref{eq:2.34} for $m=1$ implies \eqref{eq:II.18}. Since all the steps that yielded 
\eqref{eq:II.17} and \eqref{eq:II.18} from \eqref{eq:2.34} are reversible our proof of the equivalence of
Theorems \ref{thm:II.1} and \ref{thm:II.2} is now complete.
\end{proof}

\end{section}
\begin{section}{Proof of Theorem \ref{thm:2.1}}
\begin{proof} (of Theorem \ref{thm:2.1})
It was shown in \cite{GXZ:2011} and  will not be repeated here that the $*$-duals of the $C$ and $B$ operators can be computed by means of the 
following plethysic formulas
\begin{equation}
\BC_a^* P[X]\, =\, (\tttt{ -1\over q})^{a-1} 
P\big[X-\tttt{\eee M\over z}  \big]
\OM\big[\tttt{-\eee zX\over q(1-t)} \big]
\Big|_{z^{-a}},
\label{eq:3.1}
\end{equation}
\begin{equation}
\BB_a^* P[X]\, =\, P\big[X+\tttt{M\over z}  \big]
\OM\big[\tttt{-zX\over 1-t} \big]
\Big|_{z^{-a}}
\label{eq:3.2}
\end{equation}
where $P[X]$ denotes a generic symmetric polynomial and we must recall, for the benefit of the reader, that
\begin{equation}
\OM\big[\tttt{-\eee zX\over q(1-t)} \big]\, =\,  \sum_{m\ge 0}z^m e_m[\tttt{X\over q(1-t)}] 
\ess\ess\ess\ess\ess\hbox{and}\ess\ess\ess\ess\ess
\OM\big[\tttt{-  zX\over (1-t)} \big]\, =\,  \sum_{m\ge 0}(-z)^m e_m[\tttt{X\over  (1-t)}] 
\label{eq:3.3}
\end{equation}
\end{proof}

In the next few pages we will prove a collection of lemmas and propositions which, when combined, yield the identity in \eqref{eq:2.34}. We will try whenever possible to motivate the sequence of steps converging to the final result.
\sas

To start computing the left hand side of \eqref{eq:2.33} as well as  for later purposes
we need the following auxiliary  identity.
\sas
\sas

\begin{lemma} \label{lemma:3.1}
For any $a,b,c\ge 0$, we have
\begin{equation}
\nabla h_a^*e_b^*e_{c}^*
\ses
\sum_{r=0}^a\sum_{s=0}^b  
e_{a-r}[\tttt{1\over M}]h_{b-s}[\tttt{1\over M}]
(-1)^{n-r-s}
\sum_{\nu\part r+s} {e_r\big[B_\nu  \big]e_{a+b+c} \big[{XD_\nu\over M}\big]
\over w_\nu}~.
\label{eq:3.4}
\end{equation}
\end{lemma}

\begin{proof}
Setting $n=a+b+c$ and using \eqref{eq:2.23} we obtain the expansion 
\begin{equation}
\nabla h_a^*e_b^*e_{n-a-b}^*\ses
 \sum_{\mu\part n} {T_\mu\TH_\mu[X;q,t]\over w_\mu}\LL\TH_\mu\scs e_ah_bh_{n-a-b}\RR.
\label{eq:3.5}
\end{equation}
Now Proposition \ref{prop:2.5} gives
\begin{align}
\LL\TH_\mu\scs e_ah_bh_{n-a-b}\RR
&\ses 
\nabla^{-1} h_a[\tttt{X-\eee\over M}]e_b[\tttt{X-\eee\over M}]
\Big|_{X\RA D_\mu}\nonumber\\
&\ses 
\sum_{r=0}^a\sum_{s=0}^b e_{a-r}[\tttt{1\over M}]
h_{b-s}[\tttt{1\over M}]
\nabla^{-1}h_r[\tttt{X \over M}]e_s[\tttt{X \over M}]
\Big|_{X\RA D_\mu}\label{eq:3.6}
\end{align}
and again Proposition \ref{prop:2.5} gives
\begin{align}
\nabla^{-1}h_r[\tttt{X \over M}]e_s[\tttt{X \over M}]
\Big|_{X\RA D_\mu}
&= \sum_{\nu\part r+s}
{T_\nu^{-1}\TH_\nu[D_\mu;q,t]\over w_\nu}
\LL \TH_\nu\scs e_rh_s\RR\nonumber\\
\hbox{(by \eqref{eq:2.13} a) and \eqref{eq:2.24})}&=
(-1)^{n-r-s}\sum_{\nu\part r+s}
{T_\mu^{-1}\TH_\mu[D_\nu;q,t]\over w_\nu}
\ssp e_r[B_\nu]~.\label{eq:3.7}
\end{align}
Combining \eqref{eq:3.5}, \eqref{eq:3.6}  and \eqref{eq:3.7} we can thus write
\begin{align*}
\nabla h_a^*e_b^*e_{n-a-b}^*
&\ses
 \sum_{\mu\part n} {T_\mu\TH_\mu[X;q,t]\over w_\mu}\sum_{r=0}^a\sum_{s=0}^b e_{a-r}[\tttt{1\over M}]
h_{b-s}[\tttt{1\over M}]
(-1)^{n-r-s}
\sum_{\nu\part r+s}
{T_\mu^{-1}\TH_\mu[D_\nu;q,t]\over w_\nu}
e_r[B_\nu] \\
&\ses
\sum_{r=0}^a\sum_{s=0}^b  
e_{a-r}[\tttt{1\over M}]h_{b-s}[\tttt{1\over M}]
(-1)^{n-r-s}
\sum_{\nu\part r+s} {e_r\big[B_\nu  \big]\over w_\nu}
 \sum_{\mu\part n} { \TH_\mu[X;q,t]\TH_\mu[D_\nu;q,t]\over w_\mu}
\end{align*}
or better
$$
\nabla h_a^*e_b^*e_{n-a-b}^*
\ses
\sum_{r=0}^a\sum_{s=0}^b  
e_{a-r}[\tttt{1\over M}]h_{b-s}[\tttt{1\over M}]
(-1)^{n-r-s}
\sum_{\nu\part r+s} {e_r\big[B_\nu  \big]e_n\big[{XD_\nu\over M}\big]
\over w_\nu}~.
$$
This completes our proof.
\end{proof}

Our next task is to compute the following expression
\begin{equation}
\GG_{m,a,b,n}[X,q,t]\ses \BC_m^* \nabla h_a^*e_b^*e_{n-a-b}^*~.
\label{eq:3.8}
\end{equation}

In view of \eqref{eq:3.4}  we need the following auxiliary result.
\sas
\sas

\begin{lemma}\label{lemma:3.2}
\begin{equation}
\BC_m^*
  e_n\big[\tttt{XD_\nu\over M  }\big]
  \, =\, 
(-1)^{m-1}\sum_{u=m}^n e_{n-u}\big[\tttt{XD_\nu\over M}\big]
e_{u-m}\big[\tttt{   X\over  (1-t)} \big]t^{u-1 }
   M \sum_{\tau\RA \nu }c_{\nu\tau}(\tttt{T_\nu\over T_\tau})^{u-1}
   - 
\chi(m=1) e_{n-1}\big[\tttt{XD_\nu\over M}\big]
\label{eq:3.9}
\end{equation}
\end{lemma}

\begin{proof}
Recalling   that 
$$
\BC_a^* P[X]\ses(\tttt{ -1\over q})^{a-1} 
P\big[X-\tttt{\eee M\over z}  \big]
\OM\big[\tttt{-\eee zX\over q(1-t)} \big]
\Big|_{z^{-a}},
$$
we get
\begin{align*}
 \BC_m^*
  e_n\big[\tttt{XD_\nu\over M  }\big]
&=
(\tttt{ -1\over q})^{m-1}
 e_n\big[\tttt{(X-\tttt{\eee M\over z})D_\nu\over M  }\big]
 \OM\big[\tttt{-\eee zX\over q(1-t)} \big]
\Big|_{z^{-m}}\\
&=
(\tttt{ -1\over q})^{m-1} e_n\big[\big(\tttt{ X\over M}-\tttt{\eee  \over z}\big)D_\nu    \big] 
 \OM\big[\tttt{-\eee zX\over q(1-t)} \big]
\Big|_{z^{-m}}\\
 &=(\tttt{ -1\over q})^{m-1}\sum_{u=0}^n e_{n-u}\big[\tttt{XD_\nu\over M}\big]
e_{u}\big[\tttt{  -\eee D_\nu } \big] 
\OM\big[\tttt{-\eee zX\over q(1-t)} \big]
\Big|_{z^{u-m}}\\
\hbox{(by \eqref{eq:2.7}, \eqref{eq:2.10} and  \eqref{eq:3.3})} &=(\tttt{ -1\over q})^{m-1}\sum_{u=0}^n e_{n-u}\big[\tttt{XD_\nu\over M}\big]
e_{u-m}\big[\tttt{ X\over q(1-t)} \big]
h_{u}\big[\tttt{   M B_\nu-1 } \big] \\
\hbox{(by \eqref{eq:2.17})} &=(\tttt{ -1\over q})^{m-1}\sum_{u=m}^n e_{n-u}\big[\tttt{XD_\nu\over M}\big]
e_{u-m}\big[\tttt{   X\over  (1-t)} \big](q)^{m-u}
\Big( (tq)^{u-1}M \sum_{\tau\RA \nu }c_{\nu\tau}(\tttt{T_\nu\over T_\tau})^{u-1} 
\sms \chi(u=1)\Big)\\
 &=(-1)^{m-1}\sum_{u=m}^n e_{n-u}\big[\tttt{XD_\nu\over M}\big]
e_{u-m}\big[\tttt{   X\over  (1-t)} \big]t^{u-1 }
   M \sum_{\tau\RA \nu }c_{\nu\tau}(\tttt{T_\nu\over T_\tau})^{u-1}
   \sms \\
&\bigsp\bigsp\bigsp\bigsp
\sms 
(-1)^{m-1}\sum_{u=m}^n 
e_{n-u}\big[\tttt{XD_\nu\over M}\big]
e_{u-m}\big[\tttt{   X\over  (1-t)} \big] 
 \chi(u=1) ~.
\end{align*}
This equality proves \eqref{eq:3.9}.
\end{proof}

And now the first special case peals off.
\sas
\sas

\begin{prop}\label{prop:3.1}
\begin{equation}
\GG_{m,a,b,n}[X,q,t]\ses \BC_m^* \nabla h_a^*e_b^*e_{n-a-b}^*=\GG_{m,a,b,n}^{(1)}[X,q,t]
\sps \chi(m=1)\nabla h_a^*e_b^*e_{n-1-a-b}^*
\label{eq:3.10}
\end{equation}
with 
\begin{align}
\GG_{m,a,b,n}^{(1)}[X,q,t]
&\ses  
(-1)^{m-1}\sum_{r=0}^a\sum_{s=0}^b  
e_{a-r}[\tttt{1\over M}]h_{b-s}[\tttt{1\over M}]
(-1)^{n-r-s}
\ess\ess
\times\nonumber\\
&\ess\ess\ess\ess\ess\ess\ess
\times \sum_{\nu\part r+s} 
{e_r\big[B_\nu  \big]
\over w_\nu}
\sum_{u=m}^n e_{n-u}\big[\tttt{XD_\nu\over M}\big]
e_{u-m}\big[\tttt{   X\over  (1-t)} \big]t^{u-1 }
   M \sum_{\tau\RA \nu }c_{\nu\tau}(\tttt{T_\nu\over T_\tau})^{u-1}~.
\label{eq:3.11}
\end{align} 
\end{prop}

\begin{proof}
By  Lemma \ref{lemma:3.1} we get
$$
\nabla h_a^*e_b^*e_{n-a-b}^*
\ses
\sum_{r=0}^a\sum_{s=0}^b  
e_{a-r}[\tttt{1\over M}]h_{b-s}[\tttt{1\over M}]
(-1)^{n-r-s}
\sum_{\nu\part r+s} {e_r\big[B_\nu  \big]e_n\big[{XD_\nu\over M}\big]
\over w_\nu}~.
$$
Now applying $\BC_m^*$ on both sides of the equation, \eqref{eq:3.8} becomes
$$
\GG_{m,a,b,n}[X,q,t]\ses  
\sum_{r=0}^a\sum_{s=0}^b  
e_{a-r}[\tttt{1\over M}]h_{b-s}[\tttt{1\over M}]
(-1)^{n-r-s}
\sum_{\nu\part r+s} 
{e_r\big[B_\nu  \big]
\over w_\nu}
\BC_m^*e_n\big[{XD_\nu\over M}\big]~.
$$
Since Lemma \ref{lemma:3.2} gives
$$
\BC_m^*
  e_n\big[\tttt{XD_\nu\over M  }\big]
  \ses 
(-1)^{m-1}\sum_{u=m}^n e_{n-u}\big[\tttt{XD_\nu\over M}\big]
e_{u-m}\big[\tttt{   X\over  (1-t)} \big]t^{u-1 }
   M \sum_{\tau\RA \nu }c_{\nu\tau}(\tttt{T_\nu\over T_\tau})^{u-1}
   \sms 
\chi(m=1) e_{n-1}\big[\tttt{XD_\nu\over M}\big]~,
 $$
we obtain
\begin{equation}
\GG_{m,a,b,n}[X,q,t]=\GG_{m,a,b,n}^{(1)}[X,q,t]
\sps \GG_{m,a,b,n}^{(2)}[X,q,t]
\label{eq:3.12}
\end{equation}

\noindent
with 
\begin{align}
\GG_{m,a,b,n}^{(1)}[X,q,t]
&\ses  
(-1)^{m-1}\sum_{r=0}^a\sum_{s=0}^b  
e_{a-r}[\tttt{1\over M}]h_{b-s}[\tttt{1\over M}]
(-1)^{n-r-s}
\ess\ess
\times\nonumber\\
&\ess\ess\ess\ess\ess\ess\ess
\times \sum_{\nu\part r+s} 
{e_r\big[B_\nu  \big]
\over w_\nu}
\sum_{u=m}^n e_{n-u}\big[\tttt{XD_\nu\over M}\big]
e_{u-m}\big[\tttt{   X\over  (1-t)} \big]t^{u-1 }
   M \sum_{\tau\RA \nu }c_{\nu\tau}(\tttt{T_\nu\over T_\tau})^{u-1}
\label{eq:3.13}
\end{align}
and
\begin{equation}
 \GG_{m,a,b,n}^{(2)}[X,q,t]\ses  
\sms \chi(m=1)\sum_{r=0}^a\sum_{s=0}^b  
e_{a-r}[\tttt{1\over M}]h_{b-s}[\tttt{1\over M}]
(-1)^{n-1-r-s}(-1)
\sum_{\nu\part r+s} 
{e_r\big[B_\nu  \big]
\over w_\nu}
 e_{n-1}\big[\tttt{XD_\nu\over M}\big]~.
\label{eq:3.14}
\end{equation}
Now recall that Lemma \ref{lemma:3.1} gave
$$
\nabla h_a^*e_b^*e_{n-a-b}^*
\ses
\sum_{r=0}^a\sum_{s=0}^b  
e_{a-r}[\tttt{1\over M}]h_{b-s}[\tttt{1\over M}]
(-1)^{n-r-s}
\sum_{\nu\part r+s} {e_r\big[B_\nu  \big]e_n\big[{XD_\nu\over M}\big]
\over w_\nu}~.
$$
Thus, \eqref{eq:3.14} is none other than
\begin{equation}
 \GG_{m,a,b,n}^{(2)}[X,q,t]\ses  
\chi(m=1)\nabla h_a^*e_b^*e_{n-1-a-b}^*
\label{eq:3.15}
\end{equation}
and we see that \eqref{eq:3.12}, \eqref{eq:3.13} and \eqref{eq:3.14} complete our proof.
\end{proof}

\noindent
Our next task is to obtain the desired expression for the polynomial 
in \eqref{eq:3.11}. That is,
\begin{align*}
\GG_{m,a,b,n}^{(1)}[X,q,t]
&\ses  
(-1)^{m-1}\sum_{r=0}^a\sum_{s=0}^b  
e_{a-r}[\tttt{1\over M}]h_{b-s}[\tttt{1\over M}]
(-1)^{n-r-s}
\ess\ess
\times\\
&\ess\ess\ess\ess\ess\ess\ess
\times \sum_{\nu\part r+s} 
{e_r\big[B_\nu  \big]
\over w_\nu}
\sum_{u=m}^n e_{n-u}\big[\tttt{XD_\nu\over M}\big]
e_{u-m}\big[\tttt{   X\over  (1-t)} \big]t^{u-1 }
   M \sum_{\tau\RA \nu }c_{\nu\tau}(\tttt{T_\nu\over T_\tau})^{u-1}~.
\end{align*}
This will be obtained by a sequence of 
transformations. 

To begin, rearranging the order of summations  we get 
\begin{align*}
&\GG_{m,a,b,n}^{(1)}[X,q,t]=\\
&\ses  
(-1)^{m-1}\sum_{r=0}^a\sum_{s=0}^b  
e_{a-r}[\tttt{1\over M}]h_{b-s}[\tttt{1\over M}]
(-1)^{n-r-s}\sum_{u=m}^n
e_{u-m}\big[\tttt{   X\over  (1-t)} \big]t^{u-1 }
\ess\ess
\times\\
&\ess\ess\ess\ess\ess\bigsp\bigsp\bigsp
\times  
\sum_{\tau\part r+s-1}{1\over w_\tau} \sum_{\nu\LA \tau}
M
\tttt {  w_\tau 
\over w_\nu}c_{\nu\tau}(\tttt{T_\nu\over T_\tau})^{u-1}
 e_r\big[B_\nu  \big] e_{n-u}\big[\tttt{XD_\nu\over M}\big] \\
  &\ses  
(-1)^{m-1}\sum_{r=0}^a\sum_{s=0}^b  
e_{a-r}[\tttt{1\over M}]h_{b-s}[\tttt{1\over M}]
(-1)^{n-r-s}\sum_{u=m}^n
e_{u-m}\big[\tttt{   X\over  (1-t)} \big]t^{u-1 }
\ess\ess
\times\\
&\ess\ess\ess\ess\ess\hbox{(by \eqref{eq:2.16})}\ess\ess\bigsp
\times  
\sum_{\tau\part r+s-1}{1\over w_\tau} \sum_{\nu\LA \tau}
 d_{\nu\tau}(\tttt{T_\nu\over T_\tau})^{u-1}
 e_r\big[B_\tau +\tttt{T_\nu\over  T_\tau} \big] e_{n-u}
   \big[\tttt{X D_\tau  \over M}+X\tttt{T_\nu\over  T_\tau}\big]  \\
&\ses  
(-1)^{m-1}\sum_{r=0}^a\sum_{s=0}^b  
e_{a-r}[\tttt{1\over M}]h_{b-s}[\tttt{1\over M}]
(-1)^{n-r-s}\sum_{u=m}^n
e_{u-m}\big[\tttt{   X\over  (1-t)} \big]t^{u-1 }
\ess\ess
\times\\
&\ess\ess\ess\ess\ess\ess\ess
\times  
\sum_{\tau\part r+s-1}{1\over w_\tau} \sum_{\nu\LA \tau}
 d_{\nu\tau}(\tttt{T_\nu\over T_\tau})^{u-1}
\big( 
e_r\big[B_\tau\big] \sps
 e_{r-1}\big[B_\tau\big]  \tttt{T_\nu\over  T_\tau} \big)
\sum_{k=0}^{n-u}
 e_{n-u-k}  \big[\tttt{X D_\tau  \over M}\big] 
 e_{k}  \big[X\big]   (\tttt{T_\nu\over  T_\tau})^k
\end{align*}
and this is best rewritten as
\begin{align*}
&\GG_{m,a,b,n}^{(1)}[X,q,t]\\
&\ses  
(-1)^{m-1}\sum_{r=0}^a\sum_{s=0}^b  
e_{a-r}[\tttt{1\over M}]h_{b-s}[\tttt{1\over M}]
(-1)^{n-r-s}\sum_{u=m}^n
e_{u-m}\big[\tttt{   X\over  (1-t)} \big]t^{u-1 }
\ess\ess
\times\\
&\ess\ess\ess\ess\ess\ess\ess
\times  
\sum_{k=0}^{n-u}
 e_k \big[X \big] 
\sum_{\tau\part r+s-1}{1\over w_\tau}
e_{n-u-k}\big[ \tttt{X   D_\tau  \over M} \big]
 \sum_{\nu\LA \tau}
 d_{\nu\tau}(\tttt{T_\nu\over T_\tau})^{u+k-1}
\big( 
e_r\big[B_\tau\big] \sps
 e_{r-1}\big[B_\tau\big]  \tttt{T_\nu\over  T_\tau} \big)~.
\end{align*}

Notice next that the following summation interchanges
$$
\sum_{u=m}^n\sum_{k=0}^{n-u}
\ses \sum_{u=m}^n\sum_{k=0}^{n}\chi(k\le n-u)
\ses \sum_{k=0}^{n} \sum_{u=m}^{n}\chi(u+k\le n  )
$$
give 
\begin{align*}
&\GG_{m,a,b,n}^{(1)}[X,q,t]\\
&\ses  
(-t)^{m-1}\sum_{r=0}^a\sum_{s=0}^b  
e_{a-r}[\tttt{1\over M}]h_{b-s}[\tttt{1\over M}]
(-1)^{n-r-s}\sum_{k=0}^{n} \sum_{u=m}^{n}\chi(u+k\le n )
e_{u-m}\big[\tttt{   X\over  (1-t)} \big]t^{u-m }e_k \big[X \big] \\
&  
\hskip .5in \times  
\sum_{\tau\part r+s-1}{1\over w_\tau}
e_{n-u-k}\big[ \tttt{X   D_\tau  \over M} \big]
 \sum_{\nu\LA \tau}
 d_{\nu\tau}(\tttt{T_\nu\over T_\tau})^{u+k-1}
\big( 
e_r\big[B_\tau\big] \sps
 e_{r-1}\big[B_\tau\big]  \tttt{T_\nu\over  T_\tau} \big).
\end{align*}

\noindent
Making the change of variables $u\RA v=u+k$ we get
\begin{align*}
\GG_{m,a,b,n}^{(1)}[X,q,t]
&\ses  
(-t)^{m-1}\sum_{r=0}^a\sum_{s=0}^b  
e_{a-r}[\tttt{1\over M}]h_{b-s}[\tttt{1\over M}]
(-1)^{n-r-s}\sum_{k=0}^{n} \sum_{v=m+k}^{n} 
e_{v-m-k}\big[\tttt{   tX\over  (1-t)} \big] e_k \big[X \big]\\
&  
\times   
\sum_{\tau\part r+s-1}{1\over w_\tau}
e_{n-v}\big[ \tttt{X   D_\tau  \over M} \big]
 \sum_{\nu\LA \tau}
 d_{\nu\tau}(\tttt{T_\nu\over T_\tau})^{v-1}
\big( 
e_r\big[B_\tau\big] \sps
 e_{r-1}\big[B_\tau\big]  \tttt{T_\nu\over  T_\tau} \big)~.
\end{align*}
Or better yet, since
$$
\sum_{k=0}^{n} \sum_{v=m+k}^{n} 
e_{v-m-k}\big[\tttt{   tX\over  (1-t)} \big] e_k \big[X \big]
\ses
\sum_{v=m }^{n}\sum_{k=0}^{n} \chi(k\le   v-m) 
e_{v-m-k}\big[\tttt{   tX\over  (1-t)} \big] e_k \big[\tttt{(1-t)X\over  (1-t)} \big]
\ses 
 \sum_{v=m }^{n} e_{v-m}\big[\tttt{   X\over  (1-t)} \big]
$$
we have 
\begin{align*}
\GG_{m,a,b,n}^{(1)}[X,q,t]
&\ses  
(-t)^{m-1}\sum_{r=0}^a\sum_{s=0}^b  
e_{a-r}[\tttt{1\over M}]h_{b-s}[\tttt{1\over M}]
(-1)^{n-r-s}
 \sum_{v=m }^{n} e_{v-m}\big[\tttt{   X\over  (1-t)} \big]\\
&  
\times   
\sum_{\tau\part r+s-1}{1\over w_\tau}
e_{n-v}\big[ \tttt{X   D_\tau  \over M} \big]
 \sum_{\nu\LA \tau}
 d_{\nu\tau}(\tttt{T_\nu\over T_\tau})^{v-1}
\big( 
e_r\big[B_\tau\big] \sps
 e_{r-1}\big[B_\tau\big]  \tttt{T_\nu\over  T_\tau} \big)~.
\end{align*}
Now note that since $\tau\part r+s-1$ we must have $r+s-1\ge r$
for $e_r[B_\tau]\neq 0$ and $r\ge 1$  for $e_{r-1}[B_\tau]\neq 0$
We are thus brought to split $\GG_{m,a,b,n}^{(1)}[X,q,t]$ into two 
terms obtained by the corresponding variable changes
$s\RA s+1$ and $r\RA r+1$, obtaining
\begin{equation}
\GG_{m,a,b,n}^{(1)}[X,q,t]\ses \GG_{m,a,b,n}^{(1b)}[X,q,t]
\sps
\GG_{m,a,b,n}^{(1a)}[X,q,t]
\label{eq:3.16}
\end{equation}
with 
\begin{align*}
\GG_{m,a,b,n}^{(1b)}[X,q,t]
&\ses  
(-t)^{m-1}\sum_{r=0}^a\sum_{s=0}^{b-1}  
e_{a-r}[\tttt{1\over M}]h_{b-1-s}[\tttt{1\over M}]
(-1)^{n-1-r-s}
 \sum_{v=m }^{n} e_{v-m}\big[\tttt{   X\over  (1-t)} \big]
 \cr
 &  \bigsp\bigsp
\times   
\sum_{\tau\part r+s}
 {e_r [B_\tau ]  
\over w_\tau}
e_{n-v}\big[ \tttt{X   D_\tau  \over M} \big]
 \sum_{\nu\LA \tau}
 d_{\nu\tau}(\tttt{T_\nu\over T_\tau})^{v-1}
\end{align*}
and
\begin{align*}
\GG_{m,a,b,n}^{(1a)}[X,q,t]
&\ses  
(-t)^{m-1}\sum_{r=0}^{a-1}\sum_{s=0}^{b}  
e_{a-1-r}[\tttt{1\over M}]h_{b-s}[\tttt{1\over M}]
(-1)^{n-1-r-s}
 \sum_{v=m }^{n} e_{v-m}\big[\tttt{   X\over  (1-t)} \big]\\
 &  \bigsp\bigsp
\times   
\sum_{\tau\part r+s}
{e_{r} [B_\tau ]  
\over w_\tau}
e_{n-v}\big[ \tttt{X   D_\tau  \over M} \big]
 \sum_{\nu\LA \tau}
 d_{\nu\tau}(\tttt{T_\nu\over T_\tau})^{v}~.
\end{align*}

Now \eqref{eq:2.18} gives
$$
\sum_{\nu\LA \tau}d_{\nu\tau}(\tttt{T_\nu\over T_\tau})^s\ses (-1)^{s-1} e_{s-1}[D_\tau]\sps \chi(s=0)~,
$$
thus
$$
\sum_{\nu\LA \tau}d_{\nu\tau}(\tttt{T_\nu\over T_\tau})^{v-1}= (-1)^{{v-2} } e_{v-2}[D_\tau]\sps \chi(v=1)
\ess\ess\ess\hbox{and}\ess\ess\ess
\sum_{\nu\LA \tau}d_{\nu\tau}(\tttt{T_\nu\over T_\tau})^{v}= (-1)^{v-1} e_{v-1}[D_\tau]\sps \chi(v=0)~.
$$
Using these identities we get (since $v\ge m\ge 1$)
\begin{align}
\GG_{m,a,b,n}^{(1a)}[X,q,t]
&\ses  
(-t)^{m-1}\sum_{r=0}^{a-1}\sum_{s=0}^{b}  
e_{a-1-r}[\tttt{1\over M}]h_{b-s}[\tttt{1\over M}]
(-1)^{n-1-r-s}
 \sum_{v=m }^{n} e_{v-m}\big[\tttt{   X\over  (1-t)} \big]\nonumber\\
 &  \bigsp\bigsp
\times   
\sum_{\tau\part r+s}
{e_{r} [B_\tau ]  
\over w_\tau}
e_{n-v}\big[ \tttt{X   D_\tau  \over M} \big]
(-1)^{v-1} e_{v-1}[D_\tau]~.
\label{eq:3.17}
\end{align} 
On the other hand, in the previous case we have two terms
\begin{align}
&\GG_{m,a,b,n}^{(1b)}[X,q,t]\nonumber\\
&\ses  
(-t)^{m-1}\sum_{r=0}^a\sum_{s=0}^{b-1}  
e_{a-r}[\tttt{1\over M}]h_{b-1-s}[\tttt{1\over M}]
(-1)^{n-1-r-s}
 \sum_{v=m }^{n} e_{v-m}\big[\tttt{   X\over  (1-t)} \big]\nonumber\\
 &  \bigsp\bigsp \bigsp\bigsp \bigsp
\times   
\sum_{\tau\part r+s}
{e_r [B_\tau ]  
\over w_\tau}
e_{n-v}\big[ \tttt{X   D_\tau  \over M} \big]
 (-1)^{{v } } e_{v-2}[D_\tau]\ess \sps\nonumber\\
 & \ess\ess\ess \sps 
(-t)^{m-1}\sum_{r=0}^a\sum_{s=0}^{b-1}  
e_{a-r}[\tttt{1\over M}]h_{b-1-s}[\tttt{1\over M}]
(-1)^{n-1-r-s}
 \sum_{v=m }^{n} e_{v-m}\big[\tttt{   X\over  (1-t)} \big]\nonumber\\
 &  \bigsp\bigsp\bigsp\bigsp\bigsp\bigsp
\times   
\sum_{\tau\part r+s}
{e_r [B_\tau ]  
\over w_\tau}
e_{n-v}\big[ \tttt{X   D_\tau  \over M} \big]
  \chi(v=1) 
\label{eq:3.18}
\end{align} 
and since $v=1$ forces $m=1$ we  see that the second term  reduces to
$$
(-t)^{0}  \chi(m=1) 
 \sum_{r=0}^a\sum_{s=0}^{b-1}  
e_{a-r}[\tttt{1\over M}]h_{b-1-s}[\tttt{1\over M}]
(-1)^{n-1-r-s}   
\sum_{\tau\part r+s}
{e_r [B_\tau ]  
\over w_\tau}
e_{n-1}\big[ \tttt{X   D_\tau  \over M} \big]~.
$$
Recalling from Lemma \ref{lemma:3.1} that for $a+b+c=n$
$$
\nabla h_a^*e_b^*e_{c}^*
\ses
\sum_{r=0}^a\sum_{s=0}^b  
e_{a-r}[\tttt{1\over M}]h_{b-s}[\tttt{1\over M}]
(-1)^{n-r-s}
\sum_{\nu\part r+s} {e_r\big[B_\nu  \big]
\over w_\nu}e_{n}\big[\tttt{XD_\nu\over M}\big]~,
$$
we immediately recognize that the second  term is none other than
$\nabla h_a^*e_{b-1}^*e_{n-a-b}^*$. 

Thus putting together \eqref{eq:3.12}, \eqref{eq:3.15}, \eqref{eq:3.16}, \eqref{eq:3.17} and \eqref{eq:3.18},
our manipulations carried out in the last three pages have yielded the following final expression for the polynomial 
$\GG_{m,a,b,n}[X,q,t]= \BC_m^* \nabla h_a^*e_b^*e_{n-a-b}^*$.
\sas\sas

\begin{theorem} \label{thm:3.1}
For all integers $a,b\ge0$, $m\ge 1$ and $n>a+b$  we have
\begin{align*}
\BC_m^* \nabla h_a^*e_b^*e_{n-a-b}^* \ses 
t^{m-1}\Phi_{m,a,b,n}^{(1)}[X,q,t] &
\sps
t^{m-1}\Phi_{m,a,b,n}^{(2)}[X,q,t] \ess\sps\\
&
\sps \chi(m=1)\Big(\nabla h_a^*e_b^*e_{n-1-a-b}^*
\sps \nabla h_a^*e_{b-1}^*e_{n-a-b}^*\, .\big)
\end{align*}
with
\begin{align}
\Phi_{m,a,b,n}^{(1)}[X,q,t]
&\ses  
(-1)^{m-1}\sum_{r=0}^{a-1}\sum_{s=0}^{b}  
e_{a-1-r}[\tttt{1\over M}]h_{b-s}[\tttt{1\over M}]
(-1)^{n-1-r-s}
 \sum_{v=m }^{n} e_{v-m}\big[\tttt{   X\over  (1-t)} \big]\nonumber\\
 &  \bigsp\bigsp\bigsp
\times   
\sum_{\tau\part r+s}
{e_{r} [B_\tau ]  
\over w_\tau}
e_{n-v}\big[ \tttt{X   D_\tau  \over M} \big]
(-1)^{v-1} e_{v-1}[D_\tau]
\label{eq:3.19}
\end{align}
and
\begin{align}
\Phi_{m,a,b,n}^{(2)}[X,q,t]
&\ses  
(-1)^{m-1}\sum_{r=0}^a\sum_{s=0}^{b-1}  
e_{a-r}[\tttt{1\over M}]h_{b-1-s}[\tttt{1\over M}]
(-1)^{n-1-r-s}
 \sum_{v=m }^{n} e_{v-m}\big[\tttt{   X\over  (1-t)} \big]\nonumber\\
&\bigsp\bigsp\bigsp 
\times   
\sum_{\tau\part r+s}
{e_r [B_\tau ]  
\over w_\tau}
e_{n-v}\big[ \tttt{X   D_\tau  \over M} \big]
 (-1)^{{v } } e_{v-2}[D_\tau]~.
\label{eq:3.20}
\end{align}
\end{theorem}

Our next task is to identify the  sums occurring on the righthand sides of \eqref{eq:3.19} and \eqref{eq:3.20} as simply given by the following 
\sas
 
\begin{theorem} \label{thm:3.2}
\begin{equation}
a)\ess\ess \Phi^{(1)}_{m,n,a,b}\, =\,  \BB_{m-1}^*\nabla h_{a-1}^* e_b^*e_{n-a-b}^*\ess \ess\ess\ess\ess\ess
b)\ess\ess \Phi^{(2)}_{m,n,a,b}\, =\, 
\BB_{m-2}^*\nabla h_{a}^* e_{b-1}^*e_{n-1-a-b}^*~.
\label{eq:3.21}
\end{equation}
\end{theorem}

\begin{proof}
By Lemma \ref{lemma:3.1},
\begin{equation}
\nabla h_{a-1}^*e_b^*e_{n-a-b}^*
\ses
\sum_{r=0}^{a-1}\sum_{s=0}^b  
e_{a-1-r}[\tttt{1\over M}]h_{b-s}[\tttt{1\over M}]
(-1)^{n-1-r-s}
\sum_{\gg\part r+s} {e_r\big[B_\gg  \big]\over w_\gg}
 e_{n-1}\big[ \tttt{XD_\gg\over M}\big]~.
$$
Recall that we have
$$
\BB_a^* P[X]\ses P\big[X+\tttt{M\over z}  \big]\OM\big[\tttt{-zX\over 1-t} \big]
\Big|_{z^{-a}}~.
\label{eq:3.22}
\end{equation}

So to compute
 $
\BB_{m-1}^*\nabla h_{a-1}^* e_b^*e_{n-a-b}^*
$
we need
\begin{align*}
\BB_{m-1}^*e_{n-1}\big[\tttt{XD_\gg\over M}\big]
&=
e_{n-1}\big[\tttt{(X+\tttt{M\over z})D_\gg\over M}\big]
\OM\big[\tttt{-zX\over 1-t} \big]
\Big|_{z^{-(m-1)}} 
=
e_{n-1}
\big[\tttt{XD_\gg\over M} 
\sps\tttt{ D_\gg \over z}) \big]
\OM\big[\tttt{-zX\over 1-t} \big]
\Big|_{z^{-(m-1)}} \\
&\ses
\sum_{u=0}^{n-1}e_{n-1-u}
\big[\tttt{XD_\gg\over M}\big]
e_u[D_\gg] 
\OM\big[\tttt{-zX\over 1-t} \big]
\Big|_{z^{u-(m-1)}} \\
(\hbox{by \eqref{eq:3.3}})&=
\sum_{u=m  }^{n }e_{n -u}
\big[\tttt{XD_\gg\over M}\big]
e_{u-1}[D_\gg](-1)^{u-m}
e_{u-m }\big[\tttt{X\over 1-t}\big] ~.
\end{align*}
Thus
\begin{align*}
\BB_{m-1}^*\nabla h_{a-1}^*e_b^*e_{n-a-b}^*
&\ses
(-1)^{m-1}\sum_{r=0}^{a-1}\sum_{s=0}^b  
e_{a-1-r}[\tttt{1\over M}]h_{b-s}[\tttt{1\over M}]
(-1)^{n-1-r-s}\\
&\bigsp \, \times\,\sum_{\gg\part r+s} {e_r\big[B_\gg  \big]
\over w_\gg}
\sum_{u=m  }^{n }e_{n -u}
\big[\tttt{XD_\gg\over M}\big]
e_{u-1}[D_\gg](-1)^{u-1}
e_{u-m }\big[\tttt{X\over 1-t}\big]~.
\end{align*}
 This proves \eqref{eq:3.21} a).
 
 To prove \eqref{eq:3.21} b) we use  again  Lemma \ref{lemma:3.1}, 
 to write 
$$
\nabla h_{a}^*e_{b-1}^*e_{n-1-a-b}^*
\ses
\sum_{r=0}^{a}\sum_{s=0}^{b-1} e_{a-r}[\tttt{1\over M}]
h_{b-1-s}[\tttt{1\over M}]
(-1)^{n-2-r-s}\sum_{\gg\part r+s}
{e_r[B_\gg]\over  w_\gg}\ess
e_{n-2}\big[\tttt{X D_\gg\over M}\big]~.
$$
Consequently we need
 $
\BB_{m-2}^*e_{n-2}\big[\tttt{X D_\gg\over M}\big], 
 $
which,  using  \eqref{eq:3.22} again,  can be rewritten as 
\begin{align*}
&\BB_{m-2}^*e_{n-2}\big[\tttt{X D_\gg\over M}\big]\\
&=
e_{n-2}\big[\tttt{X D_\gg\over M}+\tttt{D_\gg\over z}\big]
\OM\big[\tttt{-zX\over 1-t} \big]
\Big|_{z^{-(m-2)}}
=
\sum_{u=0}^{n-2}e_{n-2-u}\big[\tttt{X D_\gg\over M}\big]
e_u[D_\gg]\OM\big[\tttt{-zX\over 1-t} \big]
\Big|_{z^{u-(m-2)}}\\
&=
\sum_{u=m-2  }^{n-2 }e_{n-2 -u}
\big[\tttt{XD_\gg\over M}\big]
e_{u}[D_\gg] 
h_{u-m+2 }\big[\tttt{-X\over 1-t}\big]
=
\sum_{u=m   }^{n  }e_{n  -u}
\big[\tttt{XD_\gg\over M}\big]
e_{u-2}[D_\gg] (-1)^{u-m}
e_{u-m  }\big[\tttt{X\over 1-t}\big]~.
\end{align*}
This gives
\begin{align*}
\BB_{m-2}^*\nabla h_{a}^* e_{b-1}^*e_{n-1-a-b}^*
&\ses
(-1)^{m-1}\sum_{r=0}^{a}\sum_{s=0}^{b-1} e_{a-r}[\tttt{1\over M}]
h_{b-1-s}[\tttt{1\over M}]
(-1)^{n-1-r-s}\\
&\ess\ess\ess\ess\ess\ess
\sum_{\gg\part r+s}
{e_r[B_\gg]\over  w_\gg}
\sum_{u=m   }^{n  }e_{n  -u}
\big[\tttt{XD_\gg\over M}\big]
e_{u-2}[D_\gg] (-1)^{u }
e_{u-m  }\big[\tttt{X\over 1-t}\big]~.
\end{align*}
This proves \eqref{eq:3.21} b) and completes our argument.
\end{proof}

Since Theorems \ref{thm:3.1} and \ref{thm:3.2} combined  give Theorem \ref{thm:II.2},
 this section is completed as  well.
\end{section}


\begin{section}{The combinatorial side}

For $\ggg \models n$ and $a+b+c=n$, let $\GG_\ggg^{a,b,c}$ denote the family of parking functions 
with diagonal composition $\ggg$ whose diagonal word is in the collection of shuffles 
$\DA A\shhh  B\shhh C$ with $A,B,C$ successive 

$
\vcenter{\hbox{\includegraphics[width=1.8in]{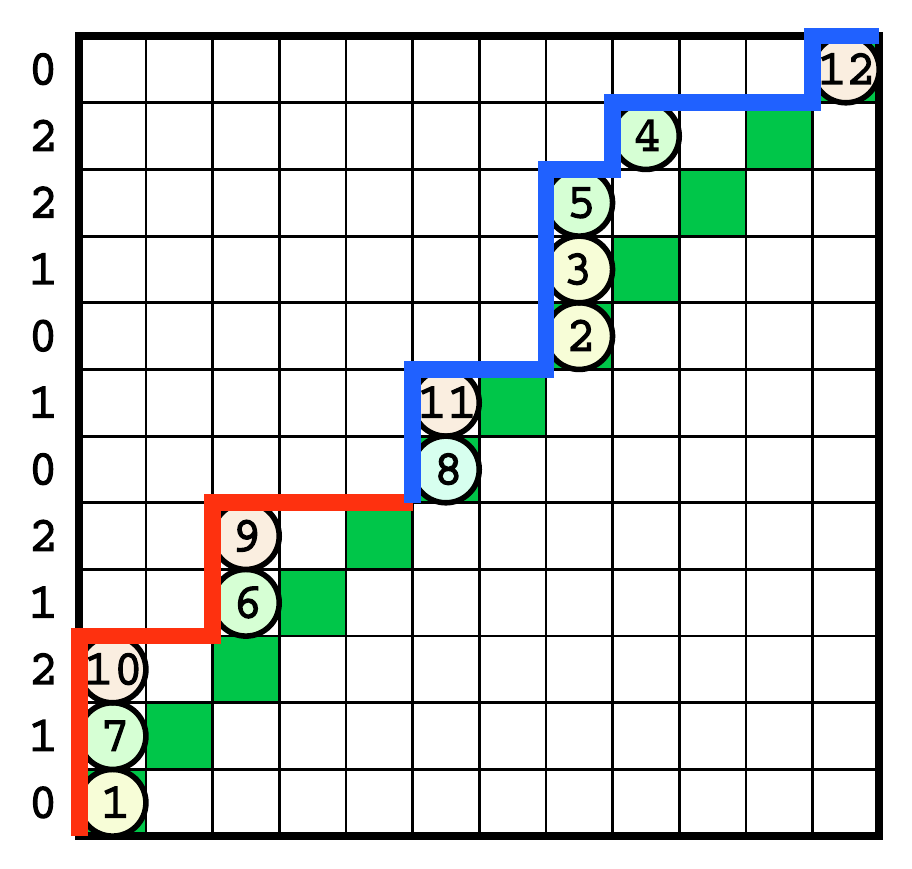}}}
$\hskip .3in
\begin{minipage}{0.62 \textwidth}
segments of the integers $1,2,\ldots ,n$ 
of respective lengths $a,b,c$. For instance in the display  below we have an element of
$\GG_{5,2,5}^{3,5,4}$. Notice that reading the cars by decreasing diagonal numbers we obtain
 $
\sig(PF)=4\, 5\, 9\, 10\, 3\, 11\, 6\, 7 \, 12\,2\, 8\, 1
 $.
which is easily seen to be a shuffle of the three words

 $$ 
\DA A=3\, 2\, 1\scs \ess\ess B=4\, 5\, 6\, 7\, 8\ess\,\hbox {and}\ess\ess 
C=9\, 10\,  11\, 12 ~.
 $$ 

 \noindent
Here and after it will be convenient to call the elements of $A$ \asmall~
cars, 
the elements of $B$ \middle~ cars
 and the elements of $C$ \abig~ cars. 
 At times, referring to these elements, we will  use the abbreviations  {\it an S},
{\it an  M } \hbox{or {\it a B } }  respectively.
\end{minipage}

The goal of  this section is to show that our family $\GG_\ggg^{a,b,c}$  satisfies the same recursion 
satisfied by  polynomials $\LL \nabla \BC_\ggg1,   e_a h_bh_{c} \RR \, $ from Theorem \ref{thm:II.2}. 
More  precisely our goal is to show that, when $a+b+c=n$ and $\aaa\models n-m$,
the polynomials

\begin{equation}
\Pi_{m,\aaa}^{a,b,c}(q,t)\ses\sum_{\multi{PF\in\CPF\cr p(PF)=(m,\aaa)}} t^{area(PF)}q^{dinv(PF)}
 \chi\big(\sig(PF)\in \,\, \DA A \shhh B\shhh C\big)
\label{eq:4.1}
\end{equation}

\vskip -.1in
\noindent
satisfy the recursion  
\begin{equation}
\hbox{\bf For $\bf m>1$:}\ess\ess \ess\ess
\Pi_{m,\aaa}^{a,b,c}(q,t)\, =\, t^{m-1} q^{\ell(\alpha)}
\Big(\sum_{\beta \models m-1} \,\,
\Pi_{\aaa,\bbb}^{a-1,b,c}(q,t)\sps
 \sum_{\beta \models m-2}   
\Pi_{\aaa,\bbb}^{a,b-1,c-1}(q,t)\Big)~,\label{eq:4.2}
\end{equation} 

\begin{align}
\hbox{\bf For $\bf m=1$:}\ess\ess \ess\ess
\Pi_{1,\aaa}^{a,b,c}(q,t)\ses
q^{\ell(\alpha)}  
&\Pi_{\aaa}^{a-1,b,c}(q,t)\sps  
\Pi_{\aaa}^{a,b-1,c}(q,t)\sps  
\Pi_{\aaa}^{a,b,c-1}(q,t)\nonumber\\
&\bigsp\bigsp 
\sps
(q-1)\sum_{i:\alpha_i = 1} q^{i-1} 
\Pi_{{\widehat \alpha}^{(i)} }^{a,b-1,c-1}(q,t)~.
\label{eq:4.3}
\end{align}
Note that this is what Theorem \ref{thm:II.2} forces us to write if Theorem \ref{thm:II.1} is to be true.
Our task is to establish  these identities using only combinatorial properties of our families
$\GG_{\gamma}^{a,b,c}$.

Before we proceed to carry this out, we need to point out a few properties of  the parking functions
in $\GG_\ggg^{a,b,c}$. To begin, note that 
since \middle~ as well as \abig~ cars are to occur in increasing order in $\sig(PF)$ 
we will never run into an $M$ on top of an $M$ nor  a $B$ on top of a $B$. Thus along  any given column of the Dyck path there is 
at most one \middle~ car and at most one \abig~ car. Clearly, on  
any column, a \middle~ car  can only lie above a \asmall~ car
 and \abig~  car  can only be on top of  a column. By contrast \asmall~ cars can be on top of \asmall~ cars.

 Note that since in the general case
$$
\DA A=a\cdots 321\scs \ess\ess\ess
B= a+1 \dots a+b\scs C=a+b+1\cdots a+b+c= n 
$$
Then if the car in cell (1,1) is an $S$ it has to be a $1$ while if in that cell we have an $M$ or a $B$
they must respectively be an $a+b$ or an $a+b+c$. For the same reason a \asmall~ car in any diagonal 
makes a primary dinv with any car in the same diagonal and to its right. While a \middle~ car can   
make a primary dinv only with a \abig~ car to its right. Likewise a \abig~ car in any diagonal 
makes a secondary dinv only with a \middle~ or a \asmall~ car to its right and in the adjacent lower diagonal. 

Keeping all this in mind we will start by proving the identity in \eqref{eq:4.2}. We do this by constructing a 
bijection $\phi$ between the family below and to the left and the union of the  families below and to the right
\begin{equation}
\GG_{m,\aaa}^{a,b,c}\ess\ess \Longleftrightarrow\ess\ess 
\sum_{\bbb\models m-1}\GG_{\aaa,\bbb}^{a-1,b,c}
\sps
\sum_{\aaa\models m-2}
  \GG_{\aaa,\bbb}^{a,b-1,c-1} .
\label{eq:4.4}
\end{equation}
Here sums denote disjoint unions. Notice that each $PF\in \GG_{m,\aaa}^{a,b,c}$ necessarily starts 
with $m$ cars, the first of which is in cell $(1,1)$ and the next $m-1$  cars are in higher diagonals. We will
refer to this car configuration as {\it the first  section of $PF$} and the following car configuration as
{\it the rest  of $PF$}. Recall  that if the first car is \asmall~ then it must be a $1$.
Next note that if the first car is not  \asmall~  then it must be $a+b$, for the first car can be  an $a+b+c$  
only when $m=1$. Moreover, on top of $a+b$, also be  a \abig~ car  for otherwise we 
are again back to the $m=1$ case.  Also notice  that, unless $m=2$, then the  third car in the first section 
is in diagonal $1$ next to the \abig~ car atop  $a+b$. To construct $\phi(PF)$ we proceed as follows 

\begin{enumerate}
\item[$(i)$] {\it If the first car is $1$ : }  Remove the main diagonal cells from the first section and cycle it to the end, as we do in the display below for $PF\in \GG_{5,2,5}^{3,5,4}$
$$
 \vcenter{\hbox{\includegraphics[width=4.5in]{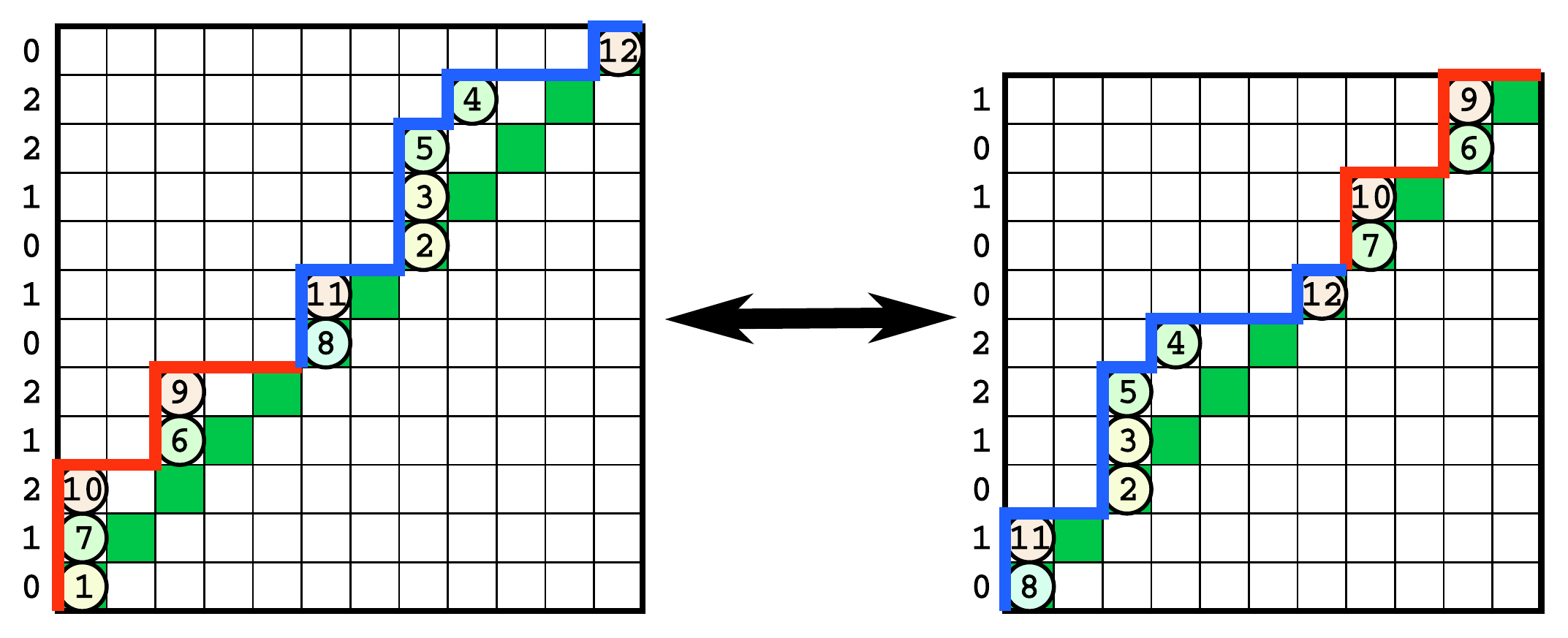}}}
$$
\item[$(ii)$] {\it If the first car is $a+b$ : } Remove the entire first column and the remaining main diagonal cells from the first section and cycle it to the end, as we do in the display below for another $PF\in \GG_{5,2,5}^{3,5,4}$
$$
 \vcenter{\hbox{\includegraphics[width=4.5in]{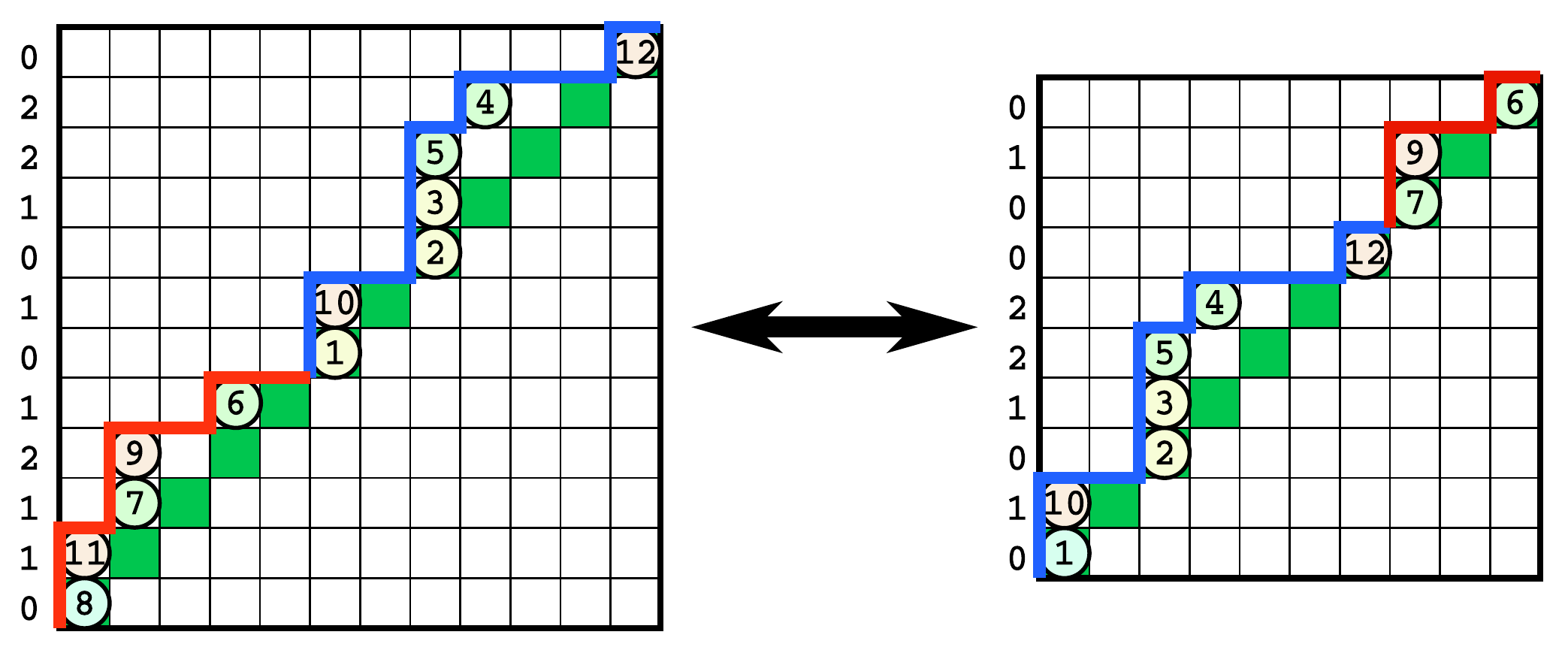}}}
$$
\item[$(iii)$] {\it If the first car is $a+b$  and $m =2$: } Remove the entire first section.
\end{enumerate}

\noindent 
To complete the construction of $\phi(PF)$,  in case $(i)$ we need only decrease by one 
all the remaining  cars numbers. In  cases $(i)$ and  $(ii)$,
if  the \abig~ car in the removed column is  a $u$,  we must  decrease by one all the \middle~ numbers   and the \abig~ numbers smaller than $u$
and finally   decrease by $2$  all the \abig~ cars numbers larger than $u$  . This given, 
notice that in case $(i)$ $\phi(PF)$ lands in  the first sum   on the right  hand side of 
\eqref{eq:4.4} while in cases $(ii)$ and $(iii)$ $\phi(PF)$ lands in  the second sum. 
Moreover since the families in \eqref{eq:4.4} are disjoint , $\phi$ has a well defined inverse, and therefore it is a bijection.
 Note further that since all the diagonal numbers of surviving cars of the first section have  decreased by one as they have been cycled,
 not only the desired shuffle condition for $\phi(PF)$ is achieved but  all  the primary dinvs between pairs of elements in different sections 
 have become secondary dinvs, and all the secondary ones  have become primary. For pairs of elements in the same section there is clearly 
 no dinv changes since their relative order has been  preserved. Remarkably in all  cases, we have the same  dinv loss of 
 $\ell(\aaa)$ units, in case $(i)$ caused  by the loss of primary dinvs  created by the removed car  $1$ and in cases $(ii)$ and $(iii)$ by the
 loss of a primary dinv created by a \abig~ car to the right of $a+b$ plus the loss of secondary dinv created by a \asmall~ or \middle~ car with the removed \abig~ car.
 Finally we note  that the lowering by one of all  the diagonal  numbers of the first section causes a loss of area of $m-1$ in   
  each of the three cases.  These observations, combined, yield us a proof of the equality in \eqref{eq:4.2}.
 \sas
 
 Our next task is to establish \eqref{eq:4.3}.  Here we have three distinct possibilities for the first car. The polynomial on the left of \eqref{eq:4.3} will  be split 
 into three parts
\begin{equation}
 \Pi_{1,\aaa}^{a,b,c}(q,t)\ses \Pi_{1,\aaa,S}^{a,b,c}(q,t)\sps \Pi_{1,\aaa,M}^{a,b,c}(q,t)\sps \Pi_{1,\aaa,B}^{a,b,c}(q,t) 
 \label{eq:4.5}
\end{equation}
 according as the first car is a  $S$, an $M$ or a $B$. We will show that

\begin{align}
&a)\ess\ess \Pi_{1,\aaa,S}^{a,b,c}(q,t)\ses q^{\ell(\alpha)}  \Pi_{\aaa}^{a-1,b,c}(q,t)
\nonumber\\
&b)\ess\ess  \Pi_{1,\aaa,M}^{a,b,c}(q,t)\ses
\Pi_{\aaa}^{a,b-1,c}(q,t)\sps
(q-1)\sum_{i:\alpha_i = 1} q^{i-1} 
\Pi_{{\widehat \alpha}^{(i)} }^{a,b-1,c-1}(q,t)\label{eq:4.6}
\\
&c)\ess\ess
\Pi_{1,\aaa,B}^{a,b,c}(q,t) 
\ses \Pi_{\aaa}^{a,b,c-1}(q,t)~.\nonumber
\end{align}

Note, that  in cases  a), b) and c) the first cars must  be  $1$, $a+b$ and $n$ respectively.
Cases a) and c) are easily dealt with. If $PF \in \GG_{1,\aaa}^{a,b,c}$ and the first car is a  $1$  we simply 
let  $\phi(PF)$ be parking function  in $ \GG_{\aaa}^{a-1,b,c}$  obtained by removing  the $1$ and decreasing by one  all remaining car numbers. 
Note that since the removed $1$ made a primary dinv with every car in the $0$-diagonal the dinv loss is precisely $\ell(\aaa)$ in this case.
This proves \eqref{eq:4.6} a).
If the first car is   $n$  we only need to  remove the first car to obtain  $\phi(PF)\in \GG_{\aaa}^{a,b,c-1}$.   
 The identity in \eqref{eq:4.6} c) then follows immediately since the removed $n$ was causing no dinv with any cars in its diagonal. 

 We are left with \eqref{eq:4.6} b). This will require a more elaborate argument. Note that removing a first car $a+b$ causes a loss of dinv
 equal to the number of \abig~ cars cars in its diagonal. The problem  is that we have to figure out  a  way of telling what that number is.
 What is remarkable, is that our work on the symmetric function side
 led us to predict  \eqref{eq:4.6} b), which, as we shall see, solves  this problem  in  a truly surprising manner.
 
 To proceed we need some definitions. Firstly, recall  that $\GG_{1,\aaa,M}^{a,b,c}$ denotes the family of   $PF\in \GG_{1,\aaa}^{a,b,c}$ whose first car is $a+b$.
 It will also be convenient to view  the parking functions 
 in $\GG_{\aaa}^{a,b-1,c}$,  for $\aaa=(\aa_1,\aaa_2,\ldots ,\aaa_k)$, as made up of  $k$ sections, the $i^{th}$ one consisting of the configuration
 of cars whose indices are in the interval
 $$
 [\aa_1+\cdots +\aaa_{i-1}\scs  \aa_1+\cdots +\aaa_i]~.
 $$
 Note further that if a \abig~ car is in the diagonal it must consist,   of a single section, since no car can be placed on top of it. 
 Thus, $\aaa$ being fixed,  any  given $PF'\in \GG_{\aaa}^{a,b-1,c}$ can  have \abig~ main diagonal cars only in the sections with indices in the set
\begin{equation}
 S(\aaa)\ses \{i\,: \, \aaa_i=1 \}~.
\label{eq:4.7}
\end{equation}
 This given, it will convenient to denote by  $S(PF')$ the subset of $S(\aa)$ giving the  indices of the sections containing the \abig~ cars in the main diagonal of $PF'$.

 Keeping this in mind, for $PF\in \GG_{1,\aaa,M}^{a,b,c}$ let 
  $PF'\in \GG_\aaa ^{a,b-1,c}$ be the parking function obtained by removing $a+b$ and lowering 
 by one all \abig~ car numbers of $PF$. If $r=|(S(PF')|$ is the number of \abig~ cars in the $0$-diagonal of $PF'$ then,  since the removed $a+b$ gave  no
contribution to the area of $PF$, it follows  that   
\begin{equation}
 t^{area(PF)}q^{dinv(PF)}\ses t^{area(PF')}q^{dinv(PF')}q^{r }~.
\label{eq:4.7p}
\end{equation}
 Guided by \eqref{eq:4.6} b) we start by writing   
$$
q^r=1+(q-1)+(q-1)q+(q-1)q^2+\cdots +(q-1)q^{r-1},
$$
so that \eqref{eq:4.7p} may be rewritten as
\begin{equation}
t^{area(PF)}q^{dinv(PF)}\, =\,
t^{area(PF')}q^{dinv(PF')}+(q-1)\sum_{s=1}^{r}t^{area(PF')}q^{dinv(PF')+s-1}.
\label{eq:4.8}
\end{equation}
Next let
\begin{equation}
S(PF')\ses \{1\le i_1<i_2<\cdots <i_r\le k\} ~.
\label{eq:4.9}
\end{equation}
This given, 
 let $PF^{(i_s)}$ be  the parking function   obtained by removing from $PF'$ the \abig~ car in section $i_s$. 
and lowering by one all the \abig~ car numbers higher the removed one. Using  our notational convention about 
${\widehat \alpha}^{(i_s)}$, we  may write
\begin{equation}
PF^{(i_s)}\in \GG_{{\widehat \alpha}^{(i_s)} }^{a,b-1,c-1}.
\label{eq:4.10}
\end{equation}

Now we have 
\begin{equation}
dinv(PF)\ses dinv(PF^{(i_s)})\sps i_s-1\sps 1\sps  r-s~.
\label{eq:4.11}
\end{equation}
To see this recall that we obtained $PF^{(i_s)}$ from $PF$ by removing car $a+b$ and the \abig~ diagonal car of section $i_s$. Now every diagonal car between these two removed cars was contributing a dinv for $PF$: the \abig~ ones with $a+b$ and and the rest of them with the removed \abig~ car. That accounts for $i_s-1$ dinvs (one for each section of $PF'$ preceding section $i_s$. An additional dinv was caused by the removed pair and finally we must not forget that the remaining \abig~ diagonal cars past section $i_s$, $r-s$ in total , created a dinv with $a+b$. 
Thus \eqref{eq:4.8} holds true precisely as stated.
Since $dinv(PF)=dinv(PF')+r$ we may rewrite \eqref{eq:4.10} as
$$
dinv(PF')+s-1\ses dinv(PF^{(i_s)})\sps i_s-1    
$$
and thus \eqref{eq:4.8} itself becomes
$$
t^{area(PF)}q^{dinv(PF)}\, =\,
t^{area(PF')}q^{dinv(PF')}+(q-1)\sum_{s=1}^{r}t^{area(PF')}q^{dinv(PF^{(i_s)})\sps i_s-1}.
$$
Since removing diagonal cars does not change areas, this may be rewritten as 
$$
t^{area(PF)}q^{dinv(PF)}\, =\,
t^{area(PF')}q^{dinv(PF')}+(q-1)\sum_{s=1}^{r}t^{area(PF^{(i_s)})}q^{dinv(PF^{(i_s)})\sps i_s-1}~,
$$
or even better, using \eqref{eq:4.9}
\begin{equation}
t^{area(PF)}q^{dinv(PF)}\, =\,
t^{area(PF')}q^{dinv(PF')}+(q-1)\sum_{i\in S(PF')} t^{area(PF^{(i)})}q^{dinv(PF^{(i )})\sps i-1}~.
\label{eq:4.12}
\end{equation}
Since the pairing $PF\leftrightarrow PF'$ is clearly a bijection of $\GG_{1,\aaa ,M}^{a,b,c}$ onto $\GG_{\aaa}^{a,b-1,c}$,
 to prove \eqref{eq:4.6} b) we need only show that summing the right hand side of \eqref{eq:4.12} over $PF'\in \GG_{\aaa}^{a,b-1,c}$ gives the right hand
 side of \eqref{eq:4.6} b). Since by definition
 $$
\sum_{PF'\in \GG_{\aaa}^{a,b-1,c}}t^{area(PF')}q^{dinv(PF')}\ses  \Pi_{\aaa}^{a,b-1,c}(q,t)
 $$
 and
 $$
\sum_{PF^{(i)}\in \GG_{{\widehat \alpha}^{(i)} }^{a,b-1,c-1}}t^{area(PF^{(i)})}q^{dinv(PF^{(i)})}\ses  
\Pi_{{\widehat \alpha}^{(i)} }^{a,b-1,c-1}(q,t)~,
 $$
we are reduced to showing that
\begin{equation}
\sum_{PF'\in \GG_{\aaa}^{a,b-1,c}} 
\sum_{i\in S(PF')} q^{i-1}\Pi_{{\widehat \alpha}^{(i)} }^{a,b-1,c-1}(q,t)
\ses  \sum_{i\in S(\aaa)} q^{i-1} 
\Pi_{{\widehat \alpha}^{(i)} }^{a,b-1,c-1}(q,t)~.
\label{eq:4.13}
\end{equation}

It develops that this identity is due to a beautiful combinatorial mechanism that  
can best be understood  by a specific example. Suppose that $\aa$ is a composition with only three parts equal to one.
In the following  figure  and on the left we display a sectionalization of the family $\GG_{\aaa}^{a,b-1,c}$ according to which 
of the three sections of length one of $\aa$ contains a \abig~ car and which does not. We depicted a singleton section with a \abig~ car by a boxed $B$ 
and a singleton section   with a \asmall~ or \middle~ car by a boxed cross. On the right of the display we have three columns depicting the
analogous sectionalization of the the three families $\GG_{{\widehat \alpha}^{(1)}}^{a,b-1,c-1}$, $\GG_{{\widehat \alpha}^{(2)}}^{a,b-1,c-1}$
and $\GG_{{\widehat \alpha}^{(3)}}^{a,b-1,c-1}$. Here the deleted section is depicted as a darkened box,. The arrows connecting the two sides indicate how the summands on the left hand side of \eqref{eq:4.13} 
are to be arranged to give the summands on the right hand side. 
$$
 \vcenter{\hbox{\includegraphics[width=5in]{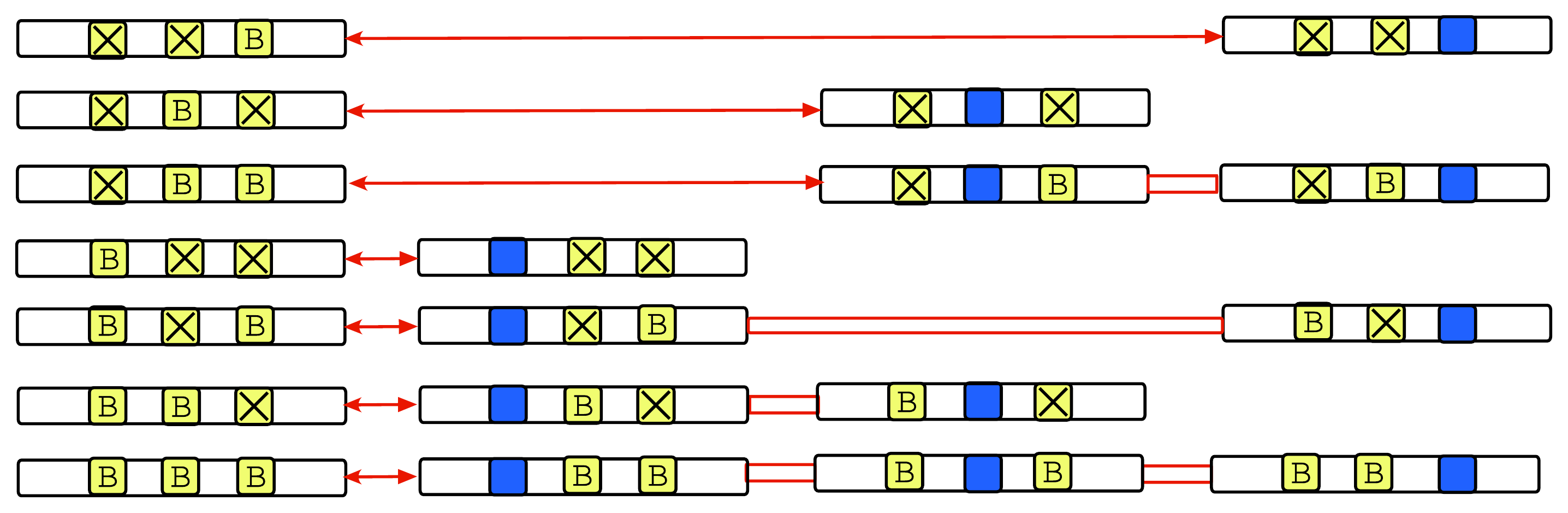}}}
$$
Notice that only the terms with a non empty $S(PF ')$ do contribute to the left hand side of \eqref{eq:4.13}. Accordingly, on the left hand side of this diagram
we have depicted the seven terms with non empty  $S(PF ')$. Notice also that the number of summands on the right hand side that correspond to a summand on the left hand side is precisely given by the size of the subset $S(PF')$. Finally as $PF'$ varies among the elements of $\GG_{\aaa}^{a,b-1,c}$ with a non empty $S(PF')$ the corresponding terms on the right hand side pick up all the sections of $\GG_{{\widehat \alpha}^{(1)}}^{a,b-1,c-1}$, $\GG_{{\widehat \alpha}^{(2)}}^{a,b-1,c-1}$
and $\GG_{{\widehat \alpha}^{(3)}}^{a,b-1,c-1}$ each exactly once an only once.
\sa

\begin{remark}
We find this last argument truly remarkable. Indeed we 
 can only wonder   how the $C$ and $B$ operators could be so savvy about the combinatorics of parking functions.
For we see here that the ultimate consequences of the innocent looking commutativity property in \eqref{eq:2.38} b) as expressed by \eqref{eq:II.18}
provide exactly what is needed to establish the equality in \eqref{eq:4.13}.
\end{remark}

We terminate our writing with a few words concerning the equality of the base cases.
To begin  we notice from \eqref{eq:4.2} and \eqref{eq:4.3} that at each iteration we loose at least one of the cars. At the moment all the \asmall~ cars are gone, the identity reduces to the $hh$
result established in \cite{GXZ}. At the moment all \middle~ or \abig~ cars are gone the identity reduces to the $eh$ result established in \cite{GXZ:2011}. If we are left with only \middle~ cars
or only \abig~ cars the identity reduces to the equality
\begin{equation}
\LL \nabla \BC_1^n \, 1\scs h_n\RR \ses \sum_{\multi{PF\in \CPF\cr p(PF)=(1^n)} }
t^{area(PF)}q^{dinv(PF)}~.
\label{eq:4.13p}
\end{equation}

Indeed, since \middle~ (or \abig) cars can't be placed on top of each other all columns of 
north steps of the supporting Dyck path must have length exactly $1$. This forces 
the cars to be all in the main diagonal . Thus $p(PF)=1^n$ and the area  vanishes. 
Since the cars must increase as we go down the diagonal the dinv vanishes as
well and the right hand side reduces to a single term equal to $1$. 
Now, using the definition in \eqref{eq:II.9},  it was shown in \cite{GXZ} (Proposition 1.3) that
$$
\BC_1\BC_1\cdots \BC_1 1\ses   q^{-{n \choose 2} }\TH_{n}[X;q,t]~,
$$
and the definition of $\nabla$ gives
\begin{equation}
\nabla \BC_1\BC_1\cdots \BC_1 1\ses  \TH_{n}[X;q,t]~.
\label{eq:4.14}
\end{equation}
Thus \eqref{eq:4.13} reduces to
$$
\LL \TH_{n}[X;q,t]\scs h_n\RR \ses 1
$$
which  is a well known equality (see \cite{GHT:1999}).   
\sas

In the case that we are left with only \asmall~ cars our equality is much deeper 
since we are essentially back to a non-trivial special  case of the $eh$ 
result.

However, we do not need to use any of our previous results, since in at most  
$n-1$ applications of our recursion we should be left with a single car, 
where   there is only one parking function with no area and no dinv,
forcing the right  hand side of \eqref{eq:II.11} to be equal to 1, while the left  hand side
reduces to $\LL \nabla \BC_1\,1 \scs e_1\RR$ or $\LL \nabla \BC_1\,1 \scs h_1\RR$
as the case may be. Both of these scalar products are trivially equal to $1$ ,
since $\BC_1\,1=h_1$ and $\nabla h_1=h_1$. 

This terminates our treatment of the combinatorial side.
\end{section}

\end{document}